\documentclass[12pt,a4paper,reqno]{amsart}
\usepackage{amsmath,amssymb,graphics,epsfig,color,enumerate,mathrsfs}
\usepackage{dsfont}  \usepackage{verbatim}
\definecolor{refkey}{gray}{.75}
\definecolor{labelkey}{gray}{.7}

\newtheorem{Theorem}{Theorem}[section]

\newtheorem{TheoremA}{Theorem}
\newtheorem{Lemma}[Theorem]{Lemma}
\newtheorem{Proposition}[Theorem]{Proposition}

\newtheorem{Remark}[Theorem]{Remark}

\newtheorem{Definition}[Theorem]{Definition}
\newtheorem{Warning}{Warning}[section]

 \definecolor{darkgreen}{rgb}{0,0.4,0}

\definecolor{light}{gray}{0.9}

\newcommand{\cC}{\ensuremath{\mathcal C}}

\newcommand{\cE}{\ensuremath{\mathcal E}}

\newcommand{\cG}{\ensuremath{\mathcal G}}

\newcommand{\cN}{\ensuremath{\mathcal N}}

\newcommand{\cR}{\ensuremath{\mathcal R}}

\newcommand{\cV}{\ensuremath{\mathcal V}}

\newcommand{\bbE}{{\ensuremath{\mathbb E}} }

\newcommand{\bbN}{{\ensuremath{\mathbb N}} }

\newcommand{\bbP}{{\ensuremath{\mathbb P}} }

\newcommand{\bbR}{{\ensuremath{\mathbb R}} }

\newcommand{\bbZ}{{\ensuremath{\mathbb Z}} }

\newcommand{\lr}{\leftrightarrow}

\newcommand{\rgh}{\rightarrow}

\let\a=\alpha \let\b=\beta   \let\d=\delta  \let\e=\varepsilon
 \let\g=\gamma       \let\l=\lambda
      \let\o=\omega      
  \let\s=\sigma \let\t=\tau   
  \let\z=\zeta
   \let\G=\Gamma  \let\L=\Lambda 
\let\O=\Omega      
\newcommand{\rosso}{\textcolor{black}}

\newcommand{\blu}{\textcolor{black}}

\newcommand{\influ}{\text{Inf\,}_i^\e}
\author[A.~Faggionato]{Alessandra Faggionato}
\address{Alessandra Faggionato.
  Dipartimento di Matematica, Universit\`a di Roma `La Sapienza'
  P.le Aldo Moro 2, 00185 Roma, Italy}
\email{faggiona@mat.uniroma1.it}

\author[H.A.~Mimun]{Hlafo Alfie Mimun}
\address{Hlafo Alfie Mimun.
  Dipartimento di Matematica, Universit\`a di Roma `La Sapienza'
  P.le Aldo Moro 2, 00185 Roma, Italy}
\email{mimun@mat.uniroma1.it}

\title[Connection probabilities]{Connection probabilities in Poisson random graphs  with uniformly bounded edges} 

\begin{document}

\begin{abstract}  We consider random graphs  with uniformly bounded edges on a Poisson point process conditioned to contain the origin. In particular we focus on  the random connection model, the  Boolean model and the Miller--Abrahams random resistor network with lower--bounded conductances. The latter is relevant for  the  analysis of   conductivity by  
Mott variable range hopping  in strongly disordered systems.  By using the method of  randomized algorithms  developed by Duminil--Copin et al.   we prove that in the subcritical phase the probability that the origin  is connected to some point  at distance $n$ decays exponentially in $n$, while in the supercritical phase the probability that the origin is connected to infinity is strictly positive and bounded from below by a term proportional to   $ (\l-\l_c)$, $\l$ being the density of the Poisson point  process and $\l_c$ being the critical density.

%

\medskip

\medskip


\thanks{This work   has been  supported  by  PRIN
  20155PAWZB ``Large Scale Random Structures". }

\end{abstract}

\maketitle

\section{Introduction}

We take   the  homogeneous Poisson point process (PPP) $\xi$ on $\bbR^d$, $d\geq 2$, with density $\l$ conditioned to contain the origin. More precisely, $\xi$ is sampled according to the Palm distribution associated to the 
homogeneous PPP    with density $\l$, which is the same as sampling a point configuration $\z$ according to the homogeneous PPP with  density $\l$ and setting $\xi:= \z \cup\{0\}$. 

We start with    two random graphs with vertex set $\xi$: the random connection model $\cG_{\rm RC}= (\xi, \cE_{\rm RC})$  with radial  \rosso{connection} function $g$ 
\cite{M}   and the Miller--Abrahams random resistor network $\cG_{\rm MA}=(\xi, \cE_{\rm MA})$ with lower--bounded  conductances (above, $\cE_{\rm RC}$ and $\cE_{\rm MA}$ denote the edge sets).

The edges in $\cE_{\rm RC}$  are determined as follows. \rosso{Recall that the connection function $g: (0,+\infty) \to [0,1]$ is a given measurable function}.   Given a realization  $\xi$,   for any unordered pair   of sites $x\not= y$ in $\xi$    one declares    $\{x,y\}$ to be an edge    (i.e.  one sets $\{x,y\} \in \cE_{\rm RC}$)  with probability $g(|x-y|)$, independently from the other pairs of sites. 
In what follows, we write $\bbP ^{\rm RC}_{0,\l}$ for the law of the above random connection model (shortly, RC model).

\medskip

We now move to the Miller--Abrahams random resistor network, explaining first  the physical motivations.    This random resistor network has been   introduced by Miller and Abrahams in \cite{MA}   as 
 an effective model to study the conductivity via Mott variable range hopping in disordered solids, as doped semiconductors, in the regime of strong Anderson localization and  low impurity density. It has been   further developed by Ambegoakar et al. \cite{AHL} to give a more robust derivation of \rosso{Mott's law}  for the low temperature asymptotics of the conductivity \cite{FM,FM_phys,FSS,POF,SE}.  Recently developed new materials, as new organic doped seminconductors, enter into this class.

The Miller--Abrahams random resistor network is obtained  as follows. Given  a realization $\xi$ of  a generic simple point process, one samples i.i.d. random variables $(E_x)_{x\in \xi}$, called \emph{energy marks}, and \rosso{attaches} to any unordered pair of sites $x\not = y$ in $\xi$ a filament of conductance \cite{AHL,POF}
\begin{equation}\label{condu}
\exp\Big\{ 
- \frac{2}{\gamma} |x-y| -\frac{\b}{2} ( |E_x|+ |E_y|+ |E_x-E_y|) 
\Big\}\,.  
\end{equation}
Above $\g$ denotes the localization length  and $\b$ the inverse temperature (in what follows we take    $\gamma=2$ and $ \b=2$ without loss of generality). Note that the skeleton of the resistor network  is the complete graph on $\xi$.  
In the physical context of inorganic doped semiconductors, the relevant distributions of the energy marks have  density function $c |E|^\a dE$ supported  on some  interval $[-a,a]$, $c$ being the normalization  constant, where  $\a\geq 0$ and $a>0$. In this case, the physical \rosso{Mott's law} states that  the conductivity scales as $\exp\{ -C \beta^{ \frac{\a+1}{\a+1+d}} \}$ for some $\b$--independent constant $C$. We refer to \cite{FM_phys} for a \rosso{conjectured} characterization of the constant $C$. 

A  key tool \rosso{(cf.~\cite{FM})} to  rigorously upper bound  the  \rosso{conductivity} of the Miller--Abrahams resistor network is provided by the control on  the size of the clusters formed by  edges with high \rosso{conductance}, when these clusters remain finite, hence in a subcritical regime.  In particular, we are interested \rosso{in} the subgraph given by the   edges   $\{x,y\}$ such that 
\begin{equation}\label{gioioso}
 |x-y| +  |E_x|+ |E_y|+ |E_x-E_y|   \leq \z\,,
 \end{equation}
for some threshold $\z>0$ for which the resulting subgraph does not percolate.

 We point out that a lower bound of the \rosso{conductivity}  would require \rosso{(cf.~\cite{FSS})} a control on the \rosso{left--right  crossings} 
in the above subgraph when it percolates    (we will address this problem in a separate work). 
\rosso{To catch the constant $C$ in Mott's law  for the Miller--Abrahams  resistor network on a Poisson point process, one needs more information on the connection probabilities and on the left--right crossings than  what  used in \cite{FM,FSS}. For  the connection probabilities   this additional information will be provided by Theorem \ref{teo1} below}.

As discussed in \cite{FM_phys}, by the  scaling properties of the model, instead of playing with $\z$ we can fix the threshold $\z$ and vary the Poisson density $\l$.

\medskip

We  now  give a self--contained mathematical definition of $\cG_{\rm MA}=(\xi, \cE_{\rm MA})$.  To this aim we fix a probability distribution  $\nu$ on $\bbR$ and a threshold $\z>0$.  Given a realization  $\xi$ of the $\l$--homogeneous PPP  conditioned to contain the origin, we consider afresh a family of i.i.d. random variables $(E_x)_{x\in \xi}$ with common distribution $\nu$. For any unordered pair   of sites $x\not= y$ in $\xi$,    we  declare     $\{x,y\}$ to be an edge    (i.e.  we set $\{x,y\} \in \cE_{\rm MA}$)  if \eqref{gioioso} is satisfied.
In what follows, we write $\bbP ^{\rm MA}_{0,\l}$ for the law of the above random graph, and we will refer to this model simply as the MA model.

 We introduce the  function $h$ defined as
\begin{equation}\label{sereno}
h(u):= P(  |E|+|E'|+|E-E'|  \leq  \z-u) \,, \qquad u\in (0,+\infty)\,,
\end{equation}
where $E,E'$ are i.i.d. random variables with law $\nu$.   \rosso{In what follows we will use the following fact:
 \begin{Lemma}\label{campanella}  The following properties are equivalent:
\begin{itemize}
\item[(i)] The function $h$ is not constantly  zero;
\item[(ii)] The probability  measure $\nu$ satisfies 
 \begin{equation}\label{giostra}
 \nu\bigl ( \,(-\z/2, \z/2) \,\bigr)>0\,.
 \end{equation}
 \end{itemize}
  \end{Lemma}
}

 The proof of Lemma \ref{campanella} is given in Section \ref{fuga}.  
 
 \smallskip
  
%
%
%
%

To state our main results we fix some notation.
We write $S_n$ for the boundary of the box $[-n,n]^d$, i.e.  $S_n=\{ x\in \bbR^d\,:\, \|x\|_\infty = n\}$ and we give  the following definition:
\begin{Definition}\label{urla}
 Given a point $x\in \bbR^d$ and given a graph $G=(V,E)$ in $\bbR^d$,  we say that $x$ is connected to $S_n$ in the graph $G$, and write $x \lr  S_n$, if  $x\in V$ and $x$ is connected in $G$ to some vertex $y\in V$ such that (i) $\|y\|_\infty \geq n$ if $\|x\|_\infty \leq n $ or (ii)   \rosso{$\|y\| _\infty \leq n$} if $\|x\|_\infty > n$. We say that a point $x\in \bbR^d$ is connected to infinity in $G$, and write $x\lr  \infty$, if  $x\in V$ and  for any $\ell >0$ there exists $y\in V $ with $\|y\|_\infty \geq \ell$  such that $x$ and $y$ are connected in $G$.
 \end{Definition}
 Both the RC model  \rosso{when   $0<\int  _0^\infty r^{d-1}g(r) dr <+\infty$} and the MA model  \rosso{when \eqref{giostra} is satisfied}  exhibit  a phase transition at some  critical density $\l_c \in (0, \infty)$:
\begin{equation}\label{PT}
\begin{cases}
 \l <\l_c \;\;  
\Longrightarrow \;\; \bbP^{RC/MA}_{0,\l} \bigl( 0   \lr    \infty\bigr)=0  \,,\\
 \l >\l_c \;\; \Longrightarrow\;\; \bbP^{RC/MA}_{0,\l} \bigl( 0   \lr  \infty\bigr) >0\,.
\end{cases}
\end{equation}
Above, and in what follows, we do not stress the dependence of the constants on the dimension $d$, the \rosso{connection} function $g$ (for the RC model), the distribution $\nu$ and the threshold $\z$ (for the MA model). 
\rosso{The above phase transition \eqref{PT} follows from  \cite[Theorem 6.1]{M} for the RC model  and from 
Proposition \ref{moldavia} in Section \ref{fuga}  for the MA model.}


\medskip


Following  the recent developments \rosso{\cite{DRT1,DRT2}} on percolation theory by means of decision trees (random algorithms) we can improve the knowledge of the above phase transition by providing more detailed information on the behavior of  the \rosso{connection} probabilities. To state our main result we need to introduce the concept of \emph{good function}:
\begin{Definition}\label{dea}
\rosso{A function $f:(0, +\infty)\to [0,1]$ is called \emph{good} if $f$ is positive on a subset of positive Lebesgue measure and  if there is a finite family of points $0 < r_1 <r_2 <\cdots < r_{m-1}<r_m$ such that (i) $f(r)=0$ for $r\geq r_m$  and (ii)
$f$ is uniformly continuous on $(r_i, r_{i+1})$ for all $i=0, \dots, m-1$, where $r_0:=0$.} 
\end{Definition}
\rosso{We point out that the function $h$ defined in \eqref{sereno}  is weakly decreasing and satisfies $h(u)=0$ for $u>\z$. In particular, due to  Lemma \ref{campanella},    $h$ is positive on a subset of positive Lebesgue measure if and only if \eqref{giostra} is satisfied.
  Moreover, due to  Lemma \ref{campanella},    if $\nu$ has a probability density which is  bounded  and which is strictly  positive on a subset of $(-\z/2, \z/2) $ of positive Lebesgue measure, 
 then the function $h$ is good. In particular, 
 if    $\nu$ has density function $c |E|^\a dE$ supported  on some  interval $[-a,a]$ (as in the physically relevant cases), then $h$ is good. }

\begin{TheoremA}\label{teo1}
Consider the random connection model $\cG_{\rm RC}$ with good radial \rosso{connection} function $g$. 
Consider   the  Miller--Abrahams model   $\cG_{\rm MA}$, associated to the distribution $\nu$ and the threshold $\z$,
and assume  that 
the function $h$ defined in \eqref{sereno} is good   (cf. Lemma \ref{campanella}). 
  In  both cases,  let the vertex set be  given by a Poisson point process with density $\l$ conditioned to contain the origin. 

Then for both models the following holds:
\begin{itemize}
\item (Subcritical phase) For any   $\l <\l_c$ there exists $c=c(\l)>0$ such that  
\begin{equation}\label{zecchino1}
\bbP_{0,\l}^{\rm RC/MA} \bigl( 0 \lr  S_n\bigr) \leq e^{- c\, n } \,, \qquad \forall n \in \bbN\,.
\end{equation}
\item (Supercritical phase)  There exists $C>0$ such that
\begin{equation}\label{zecchino2}
\bbP_{0,\l}^{\rm RC/MA} \bigl( 0 \lr  \infty \bigr) \geq C (\l -\l_c) \,, \qquad \forall \l >\l_c \,.
\end{equation}
\end{itemize}
\end{TheoremA}

\subsection{Extension to other Poisson models} We point out that the arguments presented in the proof of Theorem \ref{teo1} are robust enough to be applied to other random graphs on the Poisson point process with uniformly bounded edge length. 
We discuss  here  the Poisson Boolean model $\cG_B$ \cite{M}. Let 
  $\nu\not=\d_0$  be a  probability distribution with  bounded support in $[0,\infty)$. Given a realization $\xi$ of
   the  PPP with density $\l$ conditioned to contain the origin, let $(A_x)_{x\in \xi}$ be i.i.d. random variables with common law $\nu$.
The graph $\cG_{B}= (\xi,\cE_B)$ is  then defined  by declaring $\{x,y\}$, with $x\not =y$ in $\xi$, to be an edge  in $\cE_B$ if and only if $|x-y| \leq A_x+A_y$. It is known that the model exhibits a phase transition for some $\l_c \in (0,+\infty)$ as in \eqref{PT}.

The reader can check that  the proof of Theorem \ref{teo1} for the MA model  
  can be easily  adapted  to the Boolean model (the latter is even simpler) if one  takes now
   \begin{equation}\label{serenobis}
 h(u):=P( u\leq  A+A'
)\,, \qquad u \in (0,+\infty)\,,
\end{equation}
where $A,A'$ are i.i.d. with law $\nu$, and if one assumes  $h$ to be good. 


We collect the above observations in the following theorem:
\begin{TheoremA} Consider the  Poisson Boolean model $\cG_B$ with radius law $\nu\not=\d_0$ having  bounded support and such that the function $h$ defined in  \eqref{serenobis} is good. 
Let the  vertex set be  given by a Poisson point process with density $\l$ conditioned to contain the origin. 
Then the thesis  of Theorem \ref{teo1} remains true in this context, where $\l_c$ is the critical density for the Poisson Boolean model \cite{M}.
\end{TheoremA}

\rosso{We point out that the above result has been obtained, in part with different techniques,  in \cite{Z}.
}
%
%
%

\section{Phase transition in the MA model}\label{fuga}
In this section we prove Lemma \ref{campanella} and also show that the phase transition \eqref{PT} takes place in the MA model.

\smallskip

We start with Lemma \ref{campanella}:

\noindent
\emph{Proof of Lemma \ref{campanella}}.
Let us first show  \rosso{that Items (i) and (ii) are equivalent.}  Suppose first that \eqref{giostra} is violated and let $E,E'$ be as in \eqref{sereno}. Then a.s.  we have  $|E|\geq \z/2$ and  
$|E'|\geq \z/2$, thus implying that $P(|E|+|E'|+|E-E'| \geq \z)=1$ and therefore $h(u)=0 $ for any $u>0$. Suppose now that \eqref{giostra} is satisfied. Then it must be $\nu\bigl ([0,\z/2)\bigr)>0$ or  $\nu\bigl ((-\z/2,0]\bigr)>0$. We analyze  the first case, the other is similar.  Consider the measure $\nu_*$ given by $\nu $ restricted to $[0, \z/2)$. Let $\ell$ be the minimum of the support of $\nu_*$. Then for each $\d>0$ it holds $\nu\bigl([\ell, \ell+\d]\bigr)>0$. Since $\ell <\z/2$ we can fix  $\d>0$ such that $2 \ell +3 \d<\z$.  
 Take now $E,E'$ i.i.d. random variables with law $\nu$.
If $E,E'\in [\ell, \ell+\d]$, then $|E|+|E'|+|E-E'| \leq 2\ell+ 3\d \leq \z-u$  for any  $u>0$ such that $2 \ell +3 \d\leq \z-u$ (such a $u$ exists).   This implies that 
$h(u) \geq P\bigl( E,E'\in [\ell, \ell+\d]\bigr)=  \nu\bigl([\ell, \ell+\d]\bigr)^2>0$, hence $h$ is not constantly zero. \rosso{This completes the proof that Items (i) and (ii) are equivalent}.
\qed

\begin{Remark}\label{fiore_del_deserto}
We point out that in the above proof we have shown the following  technical fact which will be  used  in \rosso{the proof of Proposition  \ref{moldavia}}.   \rosso{If  $\nu\bigl([0,\z/2)\bigr)>0$,}  
 then  there are $\ell\geq 0$ and $\d>0$ such that (i) $ 2\ell + 3 \d<\z$, (ii)  $\nu(  [\ell, \ell+\d])>0$, (iii)   if  $e, e' \in [\ell, \ell+\d]$ then $u+|e|+|e'|+|e-e'|\leq \z$ for any $u\in(0, \z- 2\ell -3 \d]$. On the other hand, if \rosso{$\nu\bigl((-\z/2,0]\bigr)>0$},  then  there are $\ell\geq 0$ and $\d>0$ such that (i) $ 2\ell + 3 \d<\z$, (ii)  $\nu\bigl( [-\ell-\d, -\ell]\bigr)>0$, (iii)   if  $e, e' \in  [-\ell-\d, -\ell]$ then $u+|e|+|e'|+|e-e'|\leq \z$ for any $u\in(0, \z- 2\ell -3 \d]$. \rosso{Note that, due to Lemma \ref{campanella}, when $h\not \equiv 0$ the above two cases are exhaustive.}
\end{Remark}

\begin{Proposition}\label{moldavia} 
 There exists $\l_c \in (0,+\infty)$ such that the phase transition \eqref{PT} takes place in the MA model when \rosso{$h$ is not constantly  zero, equivalently when \eqref{giostra} holds (cf. Lemma \ref{campanella})}.
\end{Proposition}
The proof of the above proposition is a generalization of \rosso{the} one given in \cite{FM_phys}, in which  $\nu$ is the physically relevant distribution $\nu= c|E|^\a dE$.
\begin{proof}Since two Poisson point processes (possibly conditioned to contain the origin) with density $\l<\l'$ can be coupled  in a way that 
the one with smaller density is contained in the other, we get that the function $\phi(\l):=\bbP^{\rm MA}_{0,\l} ( 0 \lr +\infty)$ is weakly increasing.  Hence, to get the thesis it is enough to exhibit positive $\l_m, \l_M$ such that $\phi (\l_m)=0$ and $\phi(\l_M)>0$.

Let us consider the graph $\cG^*_{\rm MA} =(\xi, \cE_{\rm MA}^*)$ where a pair of sites $x\not =y$ in $\xi$ forms an edge $\{x,y\}\in \cE_{\rm MA}^*$ if and only if $|x-y|\leq \z$. Trivially,  $\cG_{\rm MA}   \subset \cG^*_{\rm MA} $. On the other hand, by the property of the Poisson Boolean model,  the event $\{ 0\lr +\infty \text{ in } \cG^*_{\rm MA} \}$ has probability zero  for  \rosso{$\l$ small enough. This proves that  $\phi(\l)=0$ for $\l$ small enough.}

Now take $\ell, \d$ as in Remark \ref{fiore_del_deserto}. We treat the case $\nu([0,\z/2))>0$, the complementary case 
$\nu ( (-\z/2,0])>0 $ is similar.
Given a realization $\xi$ of the point process and given random variables $(E_x)_{x\in\xi}$ as in the Introduction,  we build a new graph $\hat \cG_{\rm MA}=\bigl(\hat \cV_{MA},  \hat \cE_{\rm MA} \bigr) $ as follows. 
As vertex set $\hat \cV_{MA}$ we take $\{x\in \xi\,:\, E_x\in [\ell, \ell+\d]\}$.  We say that a pair of sites $x\not =y$ in $\hat \cV_{MA} $ forms an edge $\{x,y\}\in \hat \cE_{\rm MA}$ if and only if $|x-y|\leq \z-2\ell -3\d$. By Remark \ref{fiore_del_deserto} if $\{x,y\}\in \hat \cE_{\rm MA}$ then \eqref{gioioso} is satisfied, and therefore $\{x,y\} \in \cE_{\rm MA}$. We have therefore that $ \hat \cG_{\rm MA} \subset  \cG_{\rm MA}$.
On the other hand, with positive probability we have $E_0\in  [\ell, \ell+\d]$, i.e. $0\in \hat \cV_{\rm MA}$, and conditioning to this event  $\hat \cG_{\rm MA}$ becomes  a Boolean model on a PPP  with density $\l \nu (  [\ell, \ell+\d]) $ conditioned to contain the origin, where two points $x,y$ are connected by an edge if and only if $|x-y|\leq \z-2\ell -3\d$. By the properties of the Poisson Boolean model \cite{M} if $\l$ is large enough with positive probability we have $0\lr +\infty$ in  $\hat \cG_{\rm MA}$.  Since $ \hat \cG_{\rm MA} \subset  \cG_{\rm MA}$, this proves that \rosso{$\phi(\l)>0$ for $\l$ large enough}.
\end{proof}

\section{Outline of the proof of Theorem \ref{teo1}}
In this section we outline the proof of Theorem \ref{teo1}. Further details are given in the remaining sections.

\medskip

\begin{Warning}
Without loss of generality we assume, here and in what follows, that $g(r)=0$ for $r\geq 1$ in the RC model, and that  $\z<1$ in the MA model.
\end{Warning}


\subsection{Probability $\bbP_{0,\l}$ and $\bbP_\l$}\label{solitario}
We write $\cN$ for the space of possible \rosso{realizations} of a point process \rosso{in $\bbR^d$} \cite{DVJ}. We denote by $P_\l$   the law on $\cN$ of the $\l$--homogeneous Poisson point process and by  $P_{0,\l}$  the associated
Palm distribution. 
As in \cite[Sections 1.4, 1.5]{M}, given $k\in\mathbb{Z}^d$ and $n\in\mathbb{N}$, we define the binary cube of order $n$
$$\mathcal{K}(n,k):=\prod_{i=1}^d(k_i\,2^{-n},(k_i+1)\,2^{-n}].$$
Given $x\in \bbR^d $ there exists a unique binary cube of order $n$, say $\mathcal{K}(n,k(n,x))$, that contains $x$. Moreover, both for $P_\l$--a.e. $\xi$ and for $P_{0,\l}$--a.e. $\xi$,  for each $x \in \xi$  there exists a unique smallest number $n(x)$ such that $\mathcal{K}(n(x),k(n(x),x))$ contains no other point of $\xi$.

We then consider a separate  probability space $(\Sigma, P)$. For the RC model we take 
$\Sigma  = [0,1]^{\cR }$, $\cR=\bigl\{\bigl( (n_1,k_1), (n_2,k_2) \bigr) : n_1, n_2 \in \bbN  \,,\;   k_1,k_2 \in \bbZ^d\}$, and let   $P$ be the product probability measure on $\Sigma$ with marginals given by the uniform distribution on $[0,1]$.
For the MA resistor network we take 
$\Sigma  = \bbR^{\cR }$,
$\cR=\bigl\{\,(n,k), \,:\,  n \in \bbN \,,\;    k \in \bbZ^d   \}$, and let   $P$ be  the product probability measure on $\Sigma$ with marginals given by $\nu$.
Finally, we take  the following probabilities on $\cN\times \Sigma$:
\[
\bbP_\l := P_\l \times P\,, \qquad \bbP_{0,\l}:= P_{0,\l} \times P\,.
\]
We write $\s$  for a generic element of $\Sigma$. When treating the RC model,  given $x\not =y$ in $\rosso{\xi}$   we shorten the notation by writing $\s_{x,y}$ for 
 $\s_{(n_1,k_1),(n_2,k_2) }$ where 
\begin{equation}\label{Babbo}
(n_1,k_1):= \bigl(n(x), k (n(x),x)  \bigr)\,,\;(n_2,k_2):= \bigl(n(y), k (n(y),y)  \bigr)\,.
\end{equation}
Similarly, when treating the MA model, given $\rosso{x\in\xi}$ we write $\s_x$ for $\s_{n,k}$ where $(n,k)=\bigl(n(x), k (n(x),x)  \bigr)$.

\medskip

In what follows we write $\prec_{\rm lex}$ for the lexicographic order on $\bbR^d$.
To   a generic element $(\xi,\s) \in \cN \times \Sigma$  we  associate a graph $\cG=(\cV, \cE)$ defined as follows. We set $\cV:= \xi$ for the vertex set.  In the RC model 
we define the edge set   $\cE $ as the set of  pairs $\{x,y\}$ with $x \prec_{\rm lex} y$ in $\xi$ 
such that $\s_{x,y}\leq g(|x-y|)$. When treating the MA model we define  $\cE $ as the set of  pairs $\{x,y\}$ 
with $x \not= y$ 
 such that  
\[ |x-y| + |\s_x|+ |\s_y|+ |\s_x-\s_y| \leq \z\,.\]

Then the law of $\cG(\xi,\s)$ with $(\xi, \s)$ sampled according to $P_{0,\l}$ equals
$\bbP^{\rm RC}_{0,\l}$  in the RC model, while it equals $\bbP^{\rm MA}_{0,\l}$  in the MA model. In particular,  the phase transition \eqref{PT} can be stated directly for the probability $\bbP_{0,\l}$, and  to prove Theorem \ref{teo1} it is enough to
consider $\bbP_{0,\l}$ instead of $\bbP_{0,\l}^{RC/MA}$.
 Note that when $(\xi, \s)$ is sampled according to $\bbP_{\l}$, the graph $\cG(\xi,\s)$ gives a realization of $\cG_{\rm RC}/\cG_{\rm MA}$ with \rosso{the} exception that now $\xi$ is sampled according to a $\l$--homogeneous Poisson point process.

\subsection{Discretisation}\label{poligono}
We point out that, due to our assumptions, the graph $\cG$ has all  edges of length strictly smaller than $1$, both in the RC model and in the MA model.

Given \rosso{a positive integer $n$ and given} $k=0,1, \dots, n$, we define the functions
\begin{equation}
\label{solare}
  \tilde \theta_k  (\l)  := \bbP_{0,\l} \bigl ( 0 \lr  S_k   \bigr)\,, \qquad \tilde \psi_k (\l):= \l  \tilde \theta_k  (\l) \,.
\end{equation}
\begin{Warning}\label{fragola}
Above, and in what follows, we convey  that, when considering $\bbP_{0,\l}$ or the associated expectation $\bbE_{0,\l}$,  graphical statements as `` $0\lr S_k$''   refer  to the random graph $\cG$, if not stated otherwise. The same holds for $\bbP_\l$ and $\bbE_\l$.
\end{Warning}

We have $
\tilde \theta_k  (\l) =\bbP_{0,\l} \bigl ( 0 \lr  S_k   \bigr) = \bbP_\l \bigl( 0 \lr S_k\text{ in } \cG(\xi \cup\{0\}, \s)\bigr)$.
Due to \cite[Thm. 1.1]{JZG} (which remains valid when considering the additional random field $\s$), the derivative $\tilde \theta_n'(\l)$ of $\tilde \theta_n(\l)$ can be expressed as follows:
\begin{equation}\label{rana}
\tilde \theta _n ' (\l)=\frac{1}{\l} \bbE_{0,\l} \bigl[ | \text{Piv}_+ ( 0 \lr  S_n ) \setminus \{0\} | \bigr]\,,
\end{equation}
where $ \text{Piv}_+ ( 0 \lr  S_n )$ denotes the set of points which are  $(+)$--pivotal for the  event $0 \lr  S_n$. 
We recall that given an event $A$ in terms of the graph $\cG$  and a configuration $(\xi, \s) \in \cN \times \Sigma$, 
a point  $x\in \bbR^d$ is  called  \emph{$(+)$--pivotal} for the  event $A$ and the configuration $(\xi, \s)$,  if (i)  $x \in \xi$, (ii) the event $A$ takes place  for the graph  $\cG (\xi, \s)$, (iii) the event $A $ does not take  place in  the graph obtained from  $\cG ( \xi , \s)$ by removing the vertex $x$ and all edges containing $x$.

Note that $ \bbP_{0,\l} ( 0\in \text{Piv}_+ ( 0 \lr  S_n )  )=\bbP_{0,\l}  ( 0 \lr  S_n )  =\tilde \theta_n (\l)$. Hence, from \eqref{rana} we get 
\begin{equation}\label{rospo}
\tilde \psi _n'(\l)= \tilde \theta_n (\l) +\l \tilde \theta _n ' (\l) = \bbE_{0,\l} \bigl[ | \text{Piv}_+ ( 0 \lr  S_n )  | \bigr] \,.
\end{equation}




The first step in the proof of Theorem \ref{teo1}  is to  approximate  the functions $\tilde \psi_n (\l)$ and $\tilde \psi'_n (\l)$ in terms of suitable random graphs built on a grid. 
To this aim, we introduce the scale parameter $\e$  of the form  $\e=1/m$, where \rosso{$m\geq 2$ is an  integer}.  Moreover we set
\begin{align*}
& \L_k:=[-k,k)^d\,, \;S_k:=\partial \L_k= \{x \in \bbR^d\,:\, \|x\|_\infty =k \}\,;\\
& R_{x}^{\e}:=x+[0,\e)^d \text{ where } x\in \e\bbZ^d\,, \\
& \G_{\e}:=\{x\in\e\bbZ^d\,|\,R_x^{\e}\subset \L_{n+1}\}\,,  
\end{align*}
and 
\begin{equation}\label{gelsomino}
W_\e:=\begin{cases}
\bigl \{\{x,y\}\,|\,\rosso{x\not=y \text{ in }} \G_{\e}\,,  g(|x-y|)>0\} & \text{ for the RC model}\,,\\
\G_\e & \text{ for the MA  model}\,.
\end{cases}
\end{equation}

We then consider the product space $\O_\e:=\{0,1\}^{\G_\e} \times \bbR ^{W_\e}$ and write $(\eta^\e, \s^\e)$ for a generic element of $\O_\e$. We endow $\O_\e$  with the product probability measure  $\bbP^{(\e)}_\l $ making  $\eta_x^\e$, as $x $ varies in $ \G_\e$,   a Bernoulli random variable with parameter
\begin{equation}\label{pierpi}
\bbP^{(\e)}_\l (\eta_x ^\e=1)= p_\l(\e):=\frac{\l \e^d }{  1+\l \e^d}\,, 
\end{equation}
and making $\s^\e_w$, as $w$ varies in $W_\e$,  a random variable with uniform distribution on $[0,1]$ when considering the RC model, and with distribution $\nu$ when considering the MA model.
To  $( \eta^\e, \s^\e)\in \O_\e$ we associate the  graph $G_\e=(V_\e, E_\e)$ built as follows.   We set 
\begin{equation*}
V_\e:= \{ x \in \G_\e\,:\, \eta_x^\e =1 \}\,.
\end{equation*}
In the RC model    we take  
\[
  E_\e :=\bigl\{  \{x,y \} \,:\, x \not =y \text{ in } V_\e\,,\; \rosso{x \prec_{\rm lex} y} \,,\;\s^\e_{x,y} \leq g(|x-y|)\bigr \}
  \]
 and in the MA model we take 
  \[
 E_\e   := \bigl \{\{x,y\}  \,:\,  x\not= y\text{ in } V_{\e}\,,\;  |x-y| + |\s^{\e}_x|+|\s^\e_y| +|\s^\e_x-\s^\e_y| \leq \z \bigr\}\,.
\]
Given an event $A$ concerning the graph $G_\e$, we define $\text{Piv}(A)$ as the family of  sites of $\G_\e$  which are pivotal for the event $A$. More precisely, given a configuration $(\eta^\e, \s^\e)$ in $\O_\e$   and a site $x \in \G_\e$, we say that $x $  is pivotal for $A$  if 
\[ \mathds{1}_A (\eta^\e, \s^\e)\not = \mathds{1}_A (\eta^{\e,x} , \s^\e)\,,
\]
 where
$\eta^{\e,x}$ is obtained from $\eta^\e$ by replacing $\eta^\e_x$ with $1- \eta ^\e_x$. We point out  that the event $\{x\in {\rm Piv}(A)\}$ and the random variable $\eta^\e_x$ (under $\bbP^{(\e)}_\l$)  are independent.

In what follows, we write $\bbE_\l ^{(\e)}$ for the expectation associated to $\bbP_\l^{(\e)}$ and (recall Definition \ref{urla})
 we set  
 \begin{equation*}
 \tilde  \theta ^{(\e)}_k  ( \l):= \bbP^{(\e)}_\l \bigl ( 0 \lr  S_k\,|\, \eta_0^\e=1  \bigr)\,, \qquad  \theta ^{(\e)}_k  ( \l):= \bbP^{(\e)}_\l \bigl ( 0 \lr  S_k \bigr)\,.
\end{equation*}
\begin{Warning}
Above, and in what follows, we convey  that, when considering $\bbP^{(\e)}_{\l}$ or the associated expectation $\bbE^{(\e)}_{\l}$,  graphical statements as `` $0\lr S_k$''   refer  to the random graph $G_\e$, if not stated otherwise.
\end{Warning}
The following result allows to approximate  \rosso{the functions in} \eqref{solare} and \rosso{their derivatives} by their  discretized versions:
\begin{Proposition}\label{prop_discreto}
For any $n\geq 1$ and for all $k=0,1,\dots, n$ it holds
\begin{align}
& \tilde \theta_k (\l) = \lim _{\e \downarrow 0 }\tilde\theta _k ^{(\e)} (\l) \,, \label{uri1}\\
& \tilde \psi _n'(\l) = \lim _{\e \downarrow 0} \bbE^{(\e)}_\l \bigl[ \bigl|  {\rm Piv}(0 \lr  S_n ) \bigr| \bigr]\,.\label{uri2}
\end{align}
In particular, it holds $ \tilde \psi_k (\l) = \l \lim _{\e \downarrow 0 }\tilde\theta _k ^{(\e)} (\l) $.
\end{Proposition}
The last statement in Proposition \ref{prop_discreto} is an immediate consequence of \eqref{uri1}. The proof of \eqref{uri1}  is given in Section \ref{sec_discreto1}, while the proof of \eqref{uri2} is given  in Section \ref{sec_discreto2}.
\subsection{A crucial  inequality on   $\theta^{(\e)}_n(\l)$}
As explained in \cite{DRT2},   due to the phase transition \eqref{PT}, to prove Theorem \ref{teo1} it is enough to show that given $\d\in (0,1)$ there exists a positive constant $c_0= c_0 (\d)  $ such that for each $n\geq 1$
\begin{equation}\label{ermal1}
\tilde \psi_n (\l ) \leq c_0 \frac{ \sum _{k=0}^{n-1} \tilde \psi_k(\l) }{n} \tilde \psi_n '(\l)\,,  \qquad \forall \l\in [\d, \d^{-1}]\,.
\end{equation}
Indeed,   since the functions $\l \mapsto \rosso{\tilde \psi_k} (\l) $ are increasing in $\l$ and converging as $k\to \infty$, due to  \cite[Lemma 3]{DRT2} \rosso{applied to the functions $f_n(\l) := c_0\tilde \psi_n(\l)$},   \eqref{ermal1} implies that there exists $\l_* \in [\d, \d^{-1}]$ fulfilling  the following property  for any $\l \in  [\d, \d^{-1}]$: 
\begin{equation}\label{kylo_ren} 
\begin{cases}
 \l \tilde\theta_n(\l) \leq \rosso{M} e^{-c\,   n} &\text{if }\l<\l_*\text{ and }n \in \bbN\,,
\\
\l \tilde\theta(\l)\geq C(\l-\l_*)&\text{if }\l>\l_*\,,
\end{cases}
\end{equation}
where \rosso{$M=M(\d)>0$}, $C=C(\d )>0$, $c=c(\l, \d )>0$  and $\tilde\theta(\l)=\lim_{n\rgh\infty}\tilde\theta_n(\l)=\bbP_{0,\l}(0\lr \infty)$. By taking $\d$ small to have $\l_c \in [\d, \d^{-1}]$, as a byproduct of  \eqref{PT} and  \eqref{kylo_ren}  we get the identity $\l_*= \l_c$ and the thesis of Theorem \ref{teo1}. 

\bigskip

Due to Proposition \ref{prop_discreto}, we have \eqref{ermal1} if we prove that,  given $\d \in (0,1)$, there exists a positive constant $c= c (\d)  $ such that 
\begin{equation}\label{ermal2}
\tilde \theta^{(\e)}_n (\l ) \leq o(1) + c \frac{ \sum _{k=0}^{n-1} \tilde \theta^{(\e)}_k(\l) }{n}  \bbE^{(\e)}_\l \bigl[ \bigl|  {\rm Piv}(0 \lr  S_n ) \bigr| \bigr]
\end{equation}
for any $\l \in [\d, \d^{-1}]$ and $n\geq 1$, 
where the term $o(1)$ goes to zero  uniformly in 
$ \l\in [\d, \d^{-1}]$ as $\e\downarrow 0$. Since the event $\{0 \lr  S_k\} $ implies that $ \eta_0^\e=1$ and since $p_\l(\e) = O( \e^d)$  uniformly in 
$ \l\in [\d, \d^{-1}]$,  \eqref{ermal2} 
is proved whenever we show the  following proposition containing the crucial inequality on $  \theta^{(\e)}_n(\l)$:
\begin{Proposition}\label{prop_stimona}
Given $\d \in (0,1)$, there exists a positive constant $c= c (\d)  $ such that 
\begin{equation}\label{ermal3}
  \theta^{(\e)}_n (\l ) \leq o( \e^d) + c \frac{ \sum _{k=0}^{n-1}  \theta^{(\e)}_k(\l) }{n}  \bbE^{(\e)}_\l \bigl[ \bigl|  {\rm Piv}(0 \lr  S_n ) \bigr| \bigr]
\end{equation}
for any $\l \in [\d, \d^{-1}]$ and $n\geq 1$, where $o(\e^d)/\e^d$ goes to zero uniformly in  $\l \in [\d, \d^{-1}]$  as $\e\downarrow 0$.
\end{Proposition}
\subsection{Proof of Proposition  \ref{prop_stimona} by the OSSS inequality}\label{pizzarossa}
It is possible to derive \eqref{ermal3}   by applying the OSSS inequality  for product probability spaces (cf. \cite[Theorem 3.1]{OSSS}, \rosso{\cite[Remark 5]{DRT2}}).  To recall it and fix the notation in our context, we first introduce the index set $I_\e$ as the disjoint union
\[ I_\e:= \G_\e\sqcup W_\e\,.
\]
Since in the MA model $W_\e = \G_\e$, 
given $x\in \G_\e$ we write $\dot{x}$ for the site $x$ thought as element of $W_\e$ inside $I_\e$. More precisely, for the MA model  it is convenient to slightly change our notation and set $W_\e:= \{\dot{x}\,:\, x\in \G_\e\}$, thus making $W_\e$ and $\G_\e$ disjoint. We will keep the notation $\s_x^\e$, instead of $\s_{\dot{x}}^\e$, since no confusion arises.  To have a uniform notation for random variables, given $i \in I_\e$ we set 
 \[\g^\e _i:=\begin{cases}
\eta_i^{\e} &\text{if }i\in\G_\e,
\\\s_i^\e &\text{if }i\in  W_\e.
\end{cases}\]
By construction, $\g^\e=(\g^\e _i: i \in I_\e)$ is a family of  independent random variables with law $\bbP^{(\e)}_\l$.

We consider an algorithm $T$ to establish if the event $\{0\lr  S_n\}$  takes place in $G_\e$, having  input  the values $\g^\e _i$'s. At the beginning the algorithm does not reveal  (read)   all the values $\g^\e _i$'s, but it reveals some of them during the execution.   The OSSS inequality  (cf. \cite[Theorem 3.1]{OSSS}, \rosso{\cite[Remark 5]{DRT2}}) then reads 
\begin{equation}
\label{OSSS}
\text{Var}_\e (\mathds{1}_{\{0\lr  S_n\}})\leq \sum_{i\in I_\e}\d^\e _i(T) \influ ( 0\lr  S_n ),
\end{equation}
where the above variance refers to $\bbP^{(\e)}_\l$,  $\d^\e_i(T)$ and \rosso{$\influ  (0\lr  S_n)$} are respectively the \textit{revealment} and the \textit{influence} of $\g^\e_i$. More precisely, one sets 
\begin{align*}
& \d^\e _i(T):=\bbP_\l^{(\e)} (T\text{ reveals the value of }\g^\e_i)\,,\\
& \influ (\rosso{0\lr  S_n}):=\bbP_\l^{(\e)} \bigl(\mathds{1}_{\{0\lr  S_n\}}(\g ^\e )\neq\mathds{1}_{\{0\lr  S_n\}}(\g^{\e,i} )\bigr)\,,
\end{align*}
where   $ \g^{\e,i}= ( \g^{\e,i}_j\,:\, j \in I_\e) $ appearing in the second   identity is characterized by the following requirements: 
 (a) $\g^{\e,i} _j:=\g^{\e} _j$ for all $j \not =i$,  (b)  $\g^{\e,i} _i$ 
has the same distribution of $\g^\e_i$,  (c) $\g^{\e,i} _i$   is independent \rosso{of} the family $\g^\e$ (with some abuse, we have kept
the notation $\bbP_\l^{(\e)}$ for the joint law).

Since  $\text{Var}_\e(\mathds{1}_{\{0\lr  S_n\}})=\theta_n^{(\e)}(\l)(1-\theta_n^{(\e)}(\l))$, 
  \eqref{OSSS} implies for any $\e_0>0$   that 
\begin{equation}
\label{ermal4}
\theta_n^{(\e)}(\l)\leq 
c \sum_{i\in I_\e}\d^\e _i(T) \influ ( 0\lr  S_n ) \qquad \forall \e < \e_0\,,
\end{equation}
where $c:=\sup_{\l\in[\d,\d^{-1}]}\sup _{\e \leq \e_0 } (1-\theta_1^{(\e)}(\l))^{-1}$ (note that 
\rosso{$\theta^{(\e)}_n(\l) \leq \theta^{(\e)}_1(\l)$} for $n\geq 1$). As  $\rosso{\theta_1^{(\e)}(\l)} \leq \bbP^{(\e)}_\l (\eta_0^\e=1) \approx \l \e^d  $, by taking  a suitable $\e_0=\e_0(\d)$, we get that $c$ is  strictly positive and that $c$ depends only on $\d$.

Similarly to \cite{DRT2}, in order to derive \eqref{ermal3} from \eqref{ermal4}, for each $k=1, \dots, n$  we construct an algorithm $T_k$ to determine if  the event $\{0\lr  S_n\}$ occurs such that the following \rosso{Lemmas  \ref{crepes} and \ref{firenze}} are valid:
\begin{Lemma}\label{crepes}
For any $k\in \{1,2,\dots,n\}$ 
given $\d\in (0,1)$  it holds    
\begin{equation}\label{rinoceronte}
\sum_{i\in W_{\e}} \d^{\e}_i(T_k)\influ(0\lr  S_n)=  o(\e^d)\,,
\end{equation}
where $o(\e^d)/\e^d$ goes to zero uniformly in  $\l \in [\d, \d^{-1}]$  as $\e\downarrow 0$.
\end{Lemma}
\begin{Lemma} \label{firenze}
Given $\d\in (0,1)$ there exists $c=c(\d)>0$ such that, for any  $\l \in [\d,\d^{-1}]$ and any $n\geq 1$, it holds 
\begin{equation}\label{rey1}
\frac{1}{n}\sum_{k=1}^n 
\d^{\e }_i(T_k)\leq c\,\e^{-d}\frac{1}{n} \sum_{\rosso{a}=0}^{\rosso{n}-1}\theta_{\rosso{a}}^{(\e)}(\l)  \qquad \forall i \in \G_\e\,.
\end{equation}
 \end{Lemma}
The  algorithm $T_k$ is described in Section \ref{sec_leia}, while Lemmas \ref{crepes} and \ref{firenze} are proved in Section \ref{sec_jawa}.


From  \eqref{ermal4}, by averaging among $k$, we have 
\begin{equation}\label{pena}
\theta_n^{(\e)}(\l)\leq 
c   \sum_{i\in I_\e }\Big[ \frac{1}{n} \sum _{k=1}^{n}\d^{(\e)} _i(T_k)\Big] \influ ( 0\lr  S_n ) 
\end{equation}
for any $\e\leq \e_0(\d)$ and for $c=c(\d)$.
By combining \eqref{pena} with 
  Lemmas \ref{crepes} and \ref{firenze} we get 
\begin{equation}
\label{ermal7}
\theta_n^{(\e)}(\l)\leq o(\e^{d})+ c\,\e^{-d}\frac{\sum_{k=0}^{n-1}\theta_k^{(\e)}(\l)}{n}\sum_{i\in \G_\e}\influ \bigl(0\lr  S_n\bigr)
\end{equation}
for any $\e\leq \e_0(\d)$ and for $c=c(\d)$.

Hence the crucial inequality  \eqref{ermal3} in Proposition \ref{prop_stimona} follows from \eqref{ermal7}  and the following lemma:
\begin{Lemma}\label{mosca} There exists $c=c(\d)>0$ such that,  for each event $A\subset \O_\e$ which is increasing in the random variables $\eta_i^\e$'s, it holds 
\[  \influ (A) \leq c \, \e^d \bbP^{(\e)} _{\l}(i\in \text{Piv}(A)) \qquad \forall i \in \G_\e\,,\; \forall \l \in [\d,\d^{-1}]\,.\]
\end{Lemma}

The  proof of the above lemma is given in Section \ref{sec_jawa}. This concludes the proof of Proposition \ref{prop_stimona}.

\section{The algorithm $T_k$}\label{sec_leia}
Fixed  $k\in \{1,\ldots, n\}$ we are interested in constructing an algorithm $T_k$ that determines if the event $\{0\leftrightarrow S_n\} $ takes place in $G_\e$.
We  introduce  the sets 
\begin{align*}
&  L_\e=\{ \{x,y\}: x\not =y \text{ in  }\G_{\e}\,,\; f(|x-y|)>0\}\,,  \\
& \rosso{H^k_\e}=\{\{x,y\}\in L_\e \,:\,\overline{xy}\text{ intersects }S_k\}\,,
\end{align*}
where $f:=g$ in the RC model, $f:=h$ in the MA model and  $\overline{xy}$ denotes the segment in $\bbR^d$  with extremes $x,y$. For simplicity, we  set  $xy:=\{x,y\}$ with the convention that $x\prec_{\rm lex} y$.

\smallskip

We fix an ordering in $L_\e$ such that 
the elements of 
 \rosso{$H^k_\e$} precede  the elements of $L _\e\setminus \rosso{H^k_\e}$. 
Finally, we introduce the  random variables $\bigl( U^\e_{x,y}\,:\, xy \in L_\e\bigr)$ defined on $(\rosso{\O_\e}, \bbP_\l^{(\e)})$ 
as follows:
\[
U^\e_{x,y}:= 
\begin{cases}
& \mathds{1}\bigl ( \s^\e_{x,y} \leq g(|x-y|) \bigr) \text{ in the RC model}\,,\\
& \mathds{1}\bigl( |x-y| + |\s^{\e}_x|+|\s^\e_y| +|\s^\e_x-\s^\e_y| \leq \z\bigr) \text{ in the MA model}\,.
\end{cases}
\]
Note that, by definition of the edge set  $E_\e$ of the graph $G_\e$, we have that $\{x,y\}\in E_\e$ with $x\prec_{\rm lex} y$  if and only if  $xy \in L_\e$ and  $\eta^\e_x=\eta^\e_y= U_{x,y}^\e=1$. 

\smallskip

The algorithm is  organised by  \emph{meta-steps} parameterised by the elements of $L_\e$. 
 \rosso{$t(r)$} will be the number of revealed variables   up to the \rosso{$r^{th}$} meta-step included. 
 \rosso{At each meta-step the algorithm will provide two sets  $F_r, V_r$:  $V_r$ is roughly the set of vertices connected to  some edge in $E_\e \cap H^k_\e$ discovered up to the $r^{th}$ meta-step, while $F_r$ is roughly the set of edges connected to some  edge in  $E_\e \cap H^k_\e$   discovered up to the $r^{th}$ meta-step. We recall that $E_\e$ denotes the set of edges of the graph $G_\e$.
 }

\medskip
%


\centerline{\emph{Beginning of the algorithm}}

\medskip
{\bf First meta-step}.
 Let $xy$ be the first  element of $\rosso{H^k_\e}$.  \underline{Reveal}  the random variables $\eta^\e_x$ and  $\eta^\e_y$. Set $e_1:=x$, $e_2:= y$.
 
$\bullet$  If $\eta^\e_x \eta^\e_y=0$, then  set $(F_1,V_1):= (\emptyset, \emptyset)$ and $t(1)=2$, thus completing the first-meta step in this case.

$\bullet$ If $\eta^\e_x \eta^\e_y=1$, then in the RC model   \underline{reveal}  the random variable $\s^\e_{x,y}$  and set  $e_3 := xy$, $t(1):=3$, while in the MA model  \underline{reveal} the random 
 variables $\s_x^\e$, $ \s_y^\e$ and  set $e_3 :=  \dot{x}$, $e_4:= \dot{y}$, $t(1):=4$. In both cases  set 
\begin{equation}\label{ora}
(F_1,V_1):=
\begin{cases} 
(\{xy\},  \{x,y\}) &\text{ if } \rosso{U^\e _{x,y}}=1\,,
\\(\emptyset ,\emptyset )&\text{ otherwise}\,,
\end{cases}
\end{equation}
 thus completing the first meta--step in this case.
\[\text{$*$ End of the first meta--step $*$}\]

\medskip

{\bf Generic $r^{th}$ meta-step  for $r\geq 2$}.
 Distinguish two cases. 
 If $ r \leq |\rosso{H^k_\e}|$, then let $xy$ be the $r^{th}$ element of $\rosso{H^k_\e}$. If $r > |\rosso{H^k_\e}|$,  look for the minimum edge  $xy$ in $L_\e \setminus \rosso{H^k_\e}$ such that 
$\{x,y\} \cap V_{r-1}\not =\emptyset$.  
  If such an edge does not exist, then set $R_{\rm end}:= r-1$ \rosso{and}  $T_{\rm end}:=t(r-1)$, \rosso{all the generic 
  meta-steps are   completed} hence  move to the final step.

\medskip
Set $N=0$ ($N$ will play the role of counter).\\

  $\bullet$   If $\eta^\e_x$  has not been   revealed yet,  do the following:  
 \underline{reveal}  the random variable $\eta^\e_x$,  
 increase   $N$ by $+1$,
 and set $e_{t(r-1)+N}:= x$. 
 
   $\bullet$   If $\eta^\e_y$  has not been  revealed yet, then 
  \underline{reveal}  the random variable $\eta^\e_y$, increase   $N$ by $+1$ and 
  set $e_{t(r-1)+N}:= y$.

 $\bullet$  If $\eta^\e_x \eta^\e_y=0$, then  set $(F_r,V_r):= (F_{r-1}, V_{r-1})$ and $t(r):=t( r-1)+N$, 
  thus completing the $r^{th}$ meta--step in this case.

$\bullet$  If $\eta^\e_x \eta^\e_y=1$, then:

\begin{itemize}
\item[$\star$]  In the RC model    \underline{reveal}  the random variable $\s^\e_{x,y}$,   
increase   $N$ by $+1$,
 set   $e_{t(r-1)+N}:= xy$;
 \item[$\star$] In the MA model, if  $\s_x^\e$ has not been revealed yet, then  \underline{reveal} it,  
 increase   $N$ by $+1$,
    set   $e_{t(r-1)+N}:=\dot x $. In addition,   if  $\s_y^\e$ has not been revealed yet, then  \underline{reveal} it, 
     increase   $N$ by $+1$,
     set   $e_{t(r-1)+N}:= \dot y$.
     
     \smallskip
 
 In both the above $\star$--cases  set $t(r):= t(r-1) +N$, 
  \[(F_r,V_r):=
\begin{cases}
(F_r\cup\{xy\},V_r\cup\{x,y\}) &\text{ if } \rosso{U^\e _{x,y}}=1 ,
\\(F_{r-1},V_{r-1})&\text{ otherwise},
\end{cases}\]
 thus completing the $r^{th}$ meta-step.
  \end{itemize}

\medskip

{\bf Final step}. If $0\in V_{R_{\rm end}}$  and there exists a path from $0$ to 
$ V_{R_{\rm end}}\setminus (-n,n)^d$ inside  the graph $\left( V_{R_{\rm end}}, F_{R_{\rm end}}\right)$ 
 then give as output ``$0 \lr  S_n$'', otherwise give as output ``$0 \not \lr  S_n$''.

\medskip 
\centerline{\emph{End of the algorithm}}

\medskip

We conclude with some comments on the algorithm.

First, since $L_\e$ is finite, the algorithm always stops.
Moreover we note that, when the algorithm has to check if  \rosso{$U^\e _{x,y}=1$}, this is  possible using only  the revealed random variables.

By construction, in  the algorithm $T_k$, $V_{R_{\rm end}} := \{x \in \G_\e\,:\, x \lr  S_k\}$.
Moreover, $F_{R_{\rm end}}$ is the set  of edges belonging  to some path in $G_\e$ for which there is an edge $\{x,y\}$ such that the segment $\overline{xy}$ intersects $S_k$ (we shortly say that the paths intersect $S_k$).
 If $0\lr  S_n$  then there must be a path in $G_\e$ from $0$ to some point $x$  in $\G_\e \setminus (-n,n)^d $, and this path must intersect $S_k$. As a consequence, if  $0\lr  S_n$  then  there exists a path from $0$ to 
$ V_{R_{\rm end}}\setminus (-n,n)^d$ inside  the graph $\left( V_{R_{\rm end}}, F_{R_{\rm end}}\right)$. The other implication is trivially fulfilled, hence the  output of the algorithm is correct.

Finally, we point out that 
 the revealed random variables are, in chronological order, 
the ones associated to the indexes $e_1, e_2, \dots, \dots, e_{T_{\rm end}}$ (in the cases  $e_i =x$, $e_i =\dot{x}$ and  $e_i = xy$,    the associated random variables are given by  $\eta^\e _x$, $\s^\e _{x}$ and  $\s^\e_{x,y}$, respectively).

\section{Proof of Lemmas \ref{crepes}, \ref{firenze} and \ref{mosca}}\label{sec_jawa}
In this section we prove Lemmas \ref{crepes}, \ref{firenze} and \ref{mosca} which enter in the proof of Proposition \ref{prop_stimona} as discussed in Section \ref{pizzarossa}.

To simplify the notation, given $\a\in\bbR$,  we will denote by $O(\e^\a)$ any quantity  which can be bounded  from above by $C \e^\a$, where the constant $C$ can depend on $\d$ but not on the particular value  $\l \in [\d,\d^{-1}]$. Similarly, we denote by  $o(1)$ any quantity  which can be bounded  from above by $C f(\e)$, where $\lim _{\e\downarrow 0} f(\e)=0$, and both $f$ and $C$ can depend on $\d$ but not on the particular value $\l \in [\d,\d^{-1}]$. We point out that the above quantities could depend on $n$.


\subsection{Proof of Lemma \ref{crepes}}
We consider first the RC model.  Recall that in this case \rosso{$W_\e= L_\e$} (cf. \eqref{gelsomino}).
Let  $i=\{x,y\}\in L_\e$ with $x \prec _{\rm lex} y$. If $\s^\e_{x,y}$ is revealed by the algorithm, then  it must be $\eta_x^\e=\eta_y^\e=1$. Hence 
we have $\rosso{\d^{\e}_i}(T_k) \leq \bbP^{(\e)}_\l (\eta_x^\e=\eta_y^\e=1) = O(\e^{2d})$.
On the other hand,  by definition,
\begin{equation}\label{teina}
 \influ(0\lr  S_n)
=\bbP_{\l}^{(\e)}(\mathds{1}_A(\g^{\e})\neq \mathds{1}_A(\g^{\e,i})) \text{ with }A:= \{0\lr  S_n\}\,.
\end{equation}
If $\mathds{1}_A(\g^{\e})\neq \mathds{1}_A(\g^{\e,i})$ then it must be $\eta_0^\e=1, \eta_x^\e=1, \eta_y^\e=1$.
As a consequence, we get that $\influ (0\leftrightarrow S_n)
 \leq \bbP_{\l}^{(\e)}(\eta_0^\e=1, \eta_x^\e=1, \eta_y^\e=1)$.
The last probability is $O(\e^{2d})$ if the edge $\{x,y\}$ contains the origin (and there are $O(\e^{-d})$ of such edges in $W_\e$), while it is $O(\e^{3d})$ if the edge $\{x,y\}$  does not contain the origin (and there are $O(\e^{-2d})$ of such edges in $W_\e$).
Using that  $ \rosso{\d^{\e}_i}(T_k)  = O(\e^{2d})$, we get \eqref{rinoceronte}.

We now move to the MA model. Let $i=\dot{x}\in W_\e$.  If $\s^\e_{x}$ is revealed by the algorithm, then  it must be $\eta_x^\e= 1$. Hence, $\rosso{\d^{\e}_i}(T_k)  = O(\e^{d})$. On the other hand, by \eqref{teina}, if $\mathds{1}_A(\g^{\e})\neq \mathds{1}_A(\g^{\e,i})$ then it must be $\eta_0^\e=\eta_x^\e=1$.
 Hence,
$\influ (0\lr  S_n)=O(\e^d)$ if $x=0$ and $\influ (0\lr  S_n)=O(\e^{2d})$ if $x \not =0$. Since $|W_\e|=O(\e^{-d})$, we get \eqref{rinoceronte}, thus concluding the proof of Lemma \ref{crepes}.

\subsection{Proof of Lemma \ref{firenze}}\label{dim_firenze}

 In what follows, constants $c_*(d)$, $c(d)$,.. are positive constants depending only on the dimension $d$.  \rosso{We also write  $i \in H_\e^k$ if the site  $i$  belongs  to some edge in $H^k_\e$. Since the edges in $H_\e^k$ have length \blu{smaller than}  $1$, if $i\in H_\e^k$ then $d_e(i, S_k)<1$, where $d_e(\cdot, \cdot)$ denotes the Euclidean distance.
 This implies that 
$|\{ k \in \bbN: \, i \in H_\e^k\}|\leq 2$.}

 We observe that, \rosso{when $i \not \in H_\e^k$},  
 \begin{equation*}\label{shock}
  \{\rosso{\eta_i^\e} \text{ is revealed by } T_k \} \subset
   \{ i\lr  S_k \} \cup \{ \exists j\in \G_\e\setminus\{i\} : |i-j|<1 \,, \; j \lr  S_k \}
   \,.
\end{equation*}
Hence we can bound
\begin{equation}
\label{d1}
\d^\e_i (T_k)\rosso{\mathds{1}(i \not \in H^k_\e)} \leq \sum_{ j\in \G_\e : |i-j|\leq 1  } \bbP^{(\e)}_{\l}( j \lr  S_k )\,.
\end{equation}
By translation invariance 
we have $\bbP^{(\e)}_{\l}( j \lr  S_k )\leq \theta ^{(\e)} _{d(j,S_k)}(\l) $, where $d(j,S_k)$ denotes
the distance in uniform norm between $j$ and $S_k$. Note that  $d(j,S_k)\leq n$. Hence we can write
\[ \d^\e_i (T_k)\rosso{\mathds{1}(i \not \in H^k_\e)} \leq \sum_{ j\in \G_\e : |i-j|\leq 1  } \bbP^{(\e)}_{\l}( j \lr  S_k )\leq  \sum_{ j\in \G_\e : |i-j|\leq 1  } \theta ^{(\e)} _{d(j,S_k)}(\l)\,.\]
If the integer  $a\geq 0$ satisfies $a\leq d(j,S_k) \leq a+1$, then we can bound $\theta ^{(\e)} _{d(j,S_k)}(\l)\leq \theta ^{(\e)} _a(\l)$. Hence we can write
\begin{equation}\label{torino1}
 \d^\e_i (T_k) \rosso{\mathds{1}(i\not  \in H^k_\e)} \leq \sum_{a=0}^{n-1}  \theta ^{(\e)} _a(\l) 
|\{ j\in \G_\e : |i-j|\leq 1, \; a\leq d(j,S_k) \leq a+1 \}| \,.
\end{equation}
Let us now consider, for a fixed $a$,
\begin{equation}\label{torino2}
\sum_{k=1}^n  
|\{ j\in \G_\e : |i-j|\leq 1, \; a\leq d(j,S_k) \leq a+1 \}| \,.
\end{equation}
 If  $|i-j|\leq 1 $ and $ a\leq d(j,S_k) \leq a+1$, then  it must be $a- c_* \leq d(i,S_k) \leq a+ c_*$, for some constant $c_*=c_*(d)$. 
Since $k$ varies among the integers, there are  at most $c(d)$ values of $k$ for which $a- c_* \leq d(i,S_k) \leq a+ c_*$. For the other values of $k$ the associated addendum in \eqref{torino2} is simply zero. We conclude that
the sum in \eqref{torino2} is bounded by $c(d) \e^{-d}$. Therefore, averaging \eqref{torino1} among $k$, we get
\begin{equation}\label{tenacia}
\frac{1}{n}\sum_{k=1}^n \d^\e_i (T_k)  \rosso{\mathds{1}(i \not \in H^k_\e)}\leq   c(d) \e^{-d} \frac{1}{n}\sum_{a=0}^{n-1}  \theta ^{(\e)} _a(\l)  \,.
\end{equation}

\rosso{
On the other hand, by the observation made at the beginning of the proof, we can bound
$\sum_{k=1}^n \d^\e_i (T_k)  \mathds{1}(i \in H^k_\e)\leq  2$,
while 
\begin{equation}
 \e^{-d} \sum_{a=0}^{n-1}  \theta ^{(\e)} _a(\l)   \geq \e^{-d}  \theta ^{(\e)} _0(\l)
=\e^{-d} \bbP_\l ^{(\e) } (0 \lr S_0)= \e^{-d} \bbP_\l ^{(\e)} (\eta_0^\e=1)
 = \frac{\l}{1+\l \e^d}
 \,.
 \end{equation}
We therefore conclude that 
\begin{equation}\label{calmissima}
\frac{1}{n}\sum_{k=1}^n \d^\e_i (T_k)  \mathds{1}(i  \in H^k_\e)\leq   c(\d) \e^{-d} \frac{1}{n}\sum_{a=0}^{n-1}  \theta ^{(\e)} _a(\l)  \,.
\end{equation}
The thesis then follows from \eqref{tenacia} and \eqref{calmissima}.
}

\subsection{Proof of Lemma \ref{mosca}}
By symmetry we have 
\begin{equation}
\label{r2}
\begin{split}
\influ(A)& =2\bbP_{\l}^{(\e)} (\mathds{1}_A(\g^{\e})\neq \mathds{1}_A(\g^{\e,i}),\g_i^{\e}=1,\g_i^{\e,i}=0)\\
& =2\bbP_{\l}^{(\e)} (\mathds{1}_A(\g^{\e})\neq \mathds{1}_A(\hat \g^{\e}),\g_i^{\e}=1,\g_i^{\e,i}=0)\,,
\end{split}
\end{equation}
where the configuration $\hat \g^\e$ is obtained from $\g^\e$ by changing the value of $\g^\e_i=\eta^\e_i$.  The inequality 
$\mathds{1}_A(\g^{\e})\neq \mathds{1}_A(\hat \g^{\e})$ is equivalent to the fact that  $i$ is  pivotal for the event $A$ and the configuration 
$\g^\e$. Moreover, this event is independent from $\g_i^{\e}$, $\g_i^{\e,i}$. Hence \eqref{r2} implies that
\begin{equation*}
\begin{split}
 \influ (A)\leq2\bbP^{(\e)} _{\l}(i \in \rosso{\text{Piv}(A)} ,\eta_i^{\e}=1 )\leq 2\bbP^{(\e)} _{\l}(i\in \rosso{\text{Piv}(A)})p_\l(\e)\,.\end{split}
\end{equation*}
This concludes the proof of Lemma \ref{mosca}.

\section{Proof of \eqref{uri1} in Proposition \ref{prop_discreto}}\label{sec_discreto1}
In the proof below, constants $c,c_1, c_2 \dots $ are understood as positive and $\e$--independent and they can change from line to line. To simplify the notation, given $\a\in \bbR$,  we will denote by $O(\e^\a)$ any quantity  which can be bounded  from above by $C \e^\a$, where the constant $C$ can depend on $\l$. Similarly, we denote by  $o(1)$ any quantity  which can be bounded  from above by $C f(\e)$, where $\lim _{\e\downarrow 0} f(\e)=0$, and both $f$ and $C$ can depend on $\l$.  \rosso{All the above quantities can depend also on $n$, which is fixed once and for all}.

\smallskip

\rosso{Recall that $n\geq 1$}. To simplify the notation we take \blu{$k=n$} (the general case is similar). 
Recall the notation introduced in Section \ref{solitario}. We use the standard convention to identify an element $\xi$ of $\cN$ with the atomic measure $\sum_{x\in \xi}\d_x$, which will be denoted again by $\xi$. In particular, given $U\subset \bbR^d$, $\xi(U)$ equals $|\xi\cap U|$. In addition, given $\xi  \in \cN  $ and $x \in \bbR^d$, we define the translation $\t_x \xi  $ as the new set $\xi -x$.

We define the events 
\begin{align*}
& A_\e:=\bigl\{ \xi \in \cN\,:\, \xi (R_x^\e)\in \{0,1\}  \;\; \forall x \in \G_\e \bigr \}\,,\\
& B_\e:=\bigl\{ \xi\in \cN \,:\,\xi (R_0^\e)=1\bigr\}\,.
\end{align*}
If $\xi (R_x^\e)=1$,  then 
we define $\bar{x}$ as  the unique point of $\xi \cap R_{x}^{\e}$. 
On the space $\cN$ we define the functions   
\begin{equation}\label{pioggia}\varphi_x^{\e}=\mathds{1}(\xi( R_x^{\e})=1)\,,\qquad x\in \G_\e\,.
\end{equation}
Recall Warning \ref{fragola} in Section \ref{poligono}.
\begin{Lemma}\label{uffa1}
It holds
\begin{equation}\label{olga}
\tilde  \theta_n(\l)=\bbP _{0,\l}( 0 \lr   S_n )= \lim_{\e\downarrow  0}\bbP_\l (\bar{0} \lr  S_n \,|\,B_{\e}) \,.
\end{equation}
\end{Lemma} 
\begin{proof} 
We use the properties of the Campbell measure and  Palm distribution  stated in \cite[Thm. 12.2.II and Eq. (12.2.4)]{DVJ}. We apply  \cite[Eq. (12.2.4)]{DVJ} with  
\[ 
g(x, \xi):=\mathds{1}(x \in R^\e_0)\int _{\Sigma}  P(d\s) \mathds{1}\bigl(0 \leftrightarrow  S_n \text{ in } \cG(\xi,\s)\bigr ) 
\] 
(see  the notation of Section \ref{solitario}) and  get 
\begin{equation}\label{sam1} 
\begin{split}
\l \e^d \bbP_{0,\l}( 0 \lr  S_n) & = \l E_{0,\l}   \Big[ \int _{\bbR^d} \rosso{dx} \,g(x,\xi)\Big]
 = E_\l \Big[ \int _{\bbR^d} \xi (dx) g(x, \t_x \xi) \Big]\\
& = \bbE _\l \Big[ \int _{R^\e_0} \xi (dx)\mathds{1} ( x \lr    S_n(x))    \Big]
 \,,
 \end{split} \end{equation}
 where  $S_n(x):= S_n +x$.
We set $ N_\e:= \xi ( R_0^\e)$. $N_\e$ is a Poisson random variable with parameter $\l \e^d$.
We point out that 
\begin{equation}\label{pezzo2}
\begin{split}
& \bbE_\l  \Big[ \int _{R^\e_0} \xi (dx)\mathds{1} ( x \leftrightarrow  S_n(x))     \mathds{1} (N_\e \geq 2)  \Big]  \leq \bbE_\l [ N_\e    \mathds{1} (N_\e \geq 2)]\\
& = \bbE_\l[ N_\e ]- \bbP_\l ( N_\e=1) = \l \e^d(1 -  e^{- \l \e^d})= O( \e^{2d})\,.
\end{split}
\end{equation}
Moreover, 
 we can bound
\begin{equation}\label{pezzo1}
\begin{split}
\bbP_\l  ( \{ \bar 0\lr  S_{n + \e}\} \cap B_\e )& \leq 
\bbE _\l\Big[ \int _{R^\e_0} \xi (dx)\mathds{1} ( x \leftrightarrow  S_n(x)) \mathds{1} (N_\e =1)    \Big ]\\
& \leq  \bbP _\l( \{ \bar 0 \lr  S_{n -\e}\} \cap B_\e )\,.
\end{split}
\end{equation}
Since $\bbP_\l(B_\e)= \l \e^d (1+o(1) )$, from \eqref{sam1},  \eqref{pezzo2} and \eqref{pezzo1}  we conclude that 
\begin{align}
& \bbP_{0,\l}( 0 \leftrightarrow   S_n)  \geq    \bbP_\l  (  \bar 0 \lr   S_{n + \e} \,  |\, B_\e ) +o(1) \,,\label{sara1}\\
& \bbP_{0,\l}( 0 \leftrightarrow   S_n)  \leq 
  \bbP_\l  (  \bar 0\lr  S_{n - \e} \,|\, B_\e )+o(1)\,.\label{sara2}
\end{align}
On the other hand, $
\bbP_\l\bigl( \xi  ( \L_{n+\e}\setminus {\stackrel{\circ}{\L}} _{n-\e} ) \geq 1\bigr)=O(\e)$.
Since for $\e $ small (as we assume from now on)  the events $\{ \xi  ( \L_{n+\e}\setminus {\stackrel{\circ}{\L}} _{n-\e} ) \geq 1 \}$ and $B_\e$ are independent, we conclude that $\bbP_\l \bigl( \xi  ( \L_{n+\e}\setminus {\stackrel{\circ}{\L}} _{n-\e} ) \geq 1 \,|\, B_\e\bigr)= O(\e)$. 
As a consequence, we have 
\begin{equation}\label{nuvola1}
\begin{split}
 \bbP_\l ( \bar  0 \lr   S_{n  }\, |\, B_\e ) & = 
\bbP_\l  (\bar  0 \lr    S_{n }  \text{ and } 
\xi ( \L_{n+\e}\setminus {\stackrel{\circ}{\L}} _{n-\e} ) =0  | B_\e )
+o(1)  \\
& \leq
 \bbP_\l  (  \bar 0 \lr   S_{n + \e} \,  |\, B_\e )
+o(1)\,,
\end{split}
\end{equation}
and
\begin{equation}\label{nuvola2}
\begin{split}
 \bbP _\l ( \bar  0 \lr   S_{n -\e }\, |\, B_\e ) & = 
\bbP_\l  (\bar  0 \lr    S_{n-\e }  \text{ and } 
\xi ( \L_{n+\e }\setminus {\stackrel{\circ}{\L}} _{n-\e} ) =0  | B_\e )
+o(1)   \\
& \leq
 \bbP_\l  (  \bar 0 \lr   S_{n } \,  |\, B_\e )
+o(1) \,.
\end{split}
\end{equation}
By combining \eqref{sara1} with \eqref{nuvola1}, and \eqref{sara2} with \eqref{nuvola2}, we get 
\begin{equation}
\bbP_{0,\l} ( 0 \lr     S_n) = \bbP_{\l} ( \bar{0} \lr     S_n| B_\e) +o(1)\,,
\end{equation}
which is equivalent to  \eqref{olga}.
\end{proof}

\rosso{We  now enlarge   the probability space $(\Sigma, P)$ introduced in Section \ref{solitario} as follows. For the RC model the enlarged probability space is obtained from $(\Sigma, P)$ by adding independent uniform random variables $\s_{x,y} $ indexed by the 
pairs $(x,y)$ with 
 $x\not= y$ and  such that  $x,y\in \G_\e$ for some $\e=1/m$, $m $ being a positive integer.   We take these additional random variables independent from the original random variables $\s_{(n_1,k_1), (n_2,k_2)}$ defined in $(\Sigma, P)$. We point out a slight abuse of notation, since in Section \ref{solitario} we have defined $\s_{x,y}$ by means of \eqref{Babbo} when $x,y \in \xi$. On the other hand, the  probability that the realization $\xi$ of a  Poisson point process  has some vertex  in $\cup _{m=1}^\infty \G_{1/m}$ is zero, thus implying that the notation $\s_{x,y}$ is not ambiguous with probability $1$.
 For the MA model the enlarged probability space is obtained from $(\Sigma, P)$ by adding i.i.d.  random variables $\s_{x} $ with distribution $\nu$,  indexed by the points  $x$ belonging to some $ \G_\e$ as  $\e=1/m$ and  $m $ varies among the  positive integers.   Again the new random variables are independent from the ones previously defined in $(\Sigma, P)$ and  again the notation is not ambiguous  with probability $1$. To avoid new symbols, we denote by 
 $(\Sigma, P)$ the enlarged probability space and we keep the definition 
$ \bbP_\l := P_\l \times P$, $ \bbP_{0,\l}:= P_{0,\l} \times P$, where now $P$ refers to the enlarged probability space.
 }
 
 \smallskip
 
Given points $x\not =y$ and $x'\not =y'$ \rosso{we define} 
\[
\psi_{x,y}^{x',y'}:=  \begin{cases}
\mathds{1}( \s_{x,y} \leq g(|x'-y'|)) & \text{ in the RC model}\,,\\
\mathds{1}( |x'-y'| + |\s_x|+|\s_y|+|\s_x-\s_y|\leq \z) & \text{ in the MA model}\,.
\end{cases}
\]

We now introduce a new graph $\cG_\e = (\cV_\e, \cE_\e)$ which is (\rosso{as  the graph $\cG$ introduced in Section \ref{solitario}})  a function of  the pair $(\xi, \s)\in \cN \times \Sigma$.
We set 
\[\cV_\e:= \{ x\in \G_\e\,:\, \xi (R^\e_x)=1\}\,\]
while we define 
$\cE_\e$ as 
\begin{equation}\label{salto0}
\cE_{\e}=\{\{x,y\}\,|\,x,y\in \cV_\e\,,\, x\prec_{\rm lex} y
\,,\;
  \Psi _{x,y}^{ x, y }=1\}\,.
  \end{equation}
When the event $A_\e$ (defined at the beginning of the section)  takes place  
 we define a new graph $\cG_{\e}^{\#}=(\cV_{\e},\cE^{\#}_{\e})$ as function of $(\xi ,\s)\in \cN \times \Sigma$ by setting
\begin{equation}\label{salto1}
\cE^{\#}_{\e}=\{\{x,y\}\,|\,x,y\in \cV_\e\,,\, x\prec_{\rm lex} y
\,,\;
  \Psi _{x,y}^{\bar x, \bar y }=1\}\,.
  \end{equation}
Similarly, when the event $A_\e$  takes place,  
 we define a new graph $\cG_{\e}^{*}=(\cV_{\e},\cE^{*}_{\e})$ as function of $(\xi ,\s)\in \cN \times \Sigma$ by setting
\begin{equation}\label{salto2}
\cE^{*}_{\e}=\{\{x,y\}\,|\,x,y\in \cV_\e\,,\, x\prec_{\rm lex} y
\,,\;
  \Psi _{\bar x,\bar y}^{\bar x, \bar y }=1\}\,.
  \end{equation}
\rosso{We note that  $\cG _{\e}^* $ is  the graph with vertex set  $\cV_\e$  and with  edges given by the pairs  $\{x,y\}$ where $x,y $ \rosso{vary} between the sites in $\cV_\e$ with 
 $\{\bar{x}, \bar{y}\} \in \cE$. Roughly, $\cG_{\e}^*$ is the graph obtained from $\cG$ restricted to $\L_{n+1}$ by sliding the vertex at  $\bar x $  (with $x \in \G_\e$) to $x$.  
 }

Finally we observe that 
\begin{multline}\label{nocciolino}
\bbP_\l (A^c_{\e}\cap B_{\e})=\bbP_\l (B_\e \cap\{\exists y\in \G_{\e}\setminus\{0\} \,:\,\xi(R_{y}^{\e})\geq 2\})
\\   \leq  \bbP_\l (B_\e)\sum _{y\in \G_{\e}\setminus\{0\} } \bbP_\l (\xi( R_{y}^{\e})\geq 2\rosso{)}=\bbP_\l (B_\e) O( \e^{d})\,,
\end{multline}
thus implying that $\bbP_\l (A_{\e}^c\,|\,B_{\e})= O(\e^d)$.

\blu{\begin{Lemma}\label{osare}
The event $\{ \bar 0\lr S_n\text{ in }\cG\}$ equals the event  $\{  0\lr S_n\text{ in }\cG_{\e}^*\} $ if $\xi ( {\L}_{n+\e }\setminus {\stackrel{\circ}{\L}} _{n-\e} ) =0 $.
\end{Lemma}}
\begin{proof}\blu{Let  $\{a,b\}$ be  an edge of $\cG$ with $\|a\|_\infty <n$ and $\|b\|_\infty \geq n$. Since $\xi ( {\L}_{n+\e }\setminus {\stackrel{\circ}{\L}} _{n-\e} ) =0 $ we have $\|a\|_\infty < n-\e$ and $\|b\|_\infty\geq n+\e$. On the other hand, the Euclidean distance between $a$ and $b$ is smaller than $1$, thus implying that $\|a\|_\infty \geq  n+\e-1$ and  
  $\|b\|_\infty<   n+1-\e$. Suppose now that  $ \bar 0\lr S_n$ in $\cG $ and let $a,b\in \xi $ be such that $  \bar 0\lr a$, $\{a,b\}$ is an edge of $\cG$, $\|a\|_\infty <n$ and $\|b\|_\infty \geq n$. 
  As already observed, $n+\e \leq \|b\|_\infty < n+1-\e$, thus implying that $b= \bar{z}$ for some $z\in \G_\e$ and that $n\leq \|z\|_\infty< n+1$. Since $0\lr z$ in $\cG_\e^*$, we conclude that $0\lr S_n$ in  $\cG_\e^*$.}
  
  \blu{ Suppose now that $0\lr S_n$ in  $\cG_\e^*$.  Then there exists $z\in \G_\e$ such that 
    $\|z\|_\infty \geq n$, $0\lr z$ in $\cG_\e^*$. As a consequence, $\bar 0 \lr \bar z$ in $\cG$.
    Since $\xi ({\L}_{n+\e }\setminus {\stackrel{\circ}{\L}} _{n-\e} ) =0 $, it must be  $\bar{z} \in \bar{\L}_{n+\e}^c$, thus implying that $\bar{0}\lr S_n$ in $\cG$.}
      \end{proof}

\begin{Lemma}\label{uffa2} It holds 
\begin{equation}\label{rondine}
\bbP_\l (\bar 0\lr  S_n \,|\,B_{\e})=\bbP_\l (0\lr   S_n\text{ in }\cG_{\e}\,|\,A_\e\cap B_{\e})+o(1)\,.
\end{equation}
\end{Lemma}
\begin{proof}
Since $\bbP_\l (A_{\e}^c\,|\,B_{\e})= O(\e^d)$ we can write 
\begin{align}
& \bbP_\l (\bar 0\lr S_n \,|\,B_{\e}) = 
\bbP_\l (  \bar 0\lr S_n\, |\,  A_\e \cap B_{\e}) +O(\e^{d})\,.\label{57} 
\end{align}
From now on we suppose the event $A_\e \cap B_{\e}$ to take place. 
\blu{We want to apply Lemma \ref{osare}. By independence, $ \bbP_\l \bigl( \{\xi ({ \L}_{n+\e }\setminus {\stackrel{\circ}{\L}} _{n-\e} ) \geq 1 \}\cap B_{\e} \bigr)= O (\e^{d+1})$,  while $\bbP_\l(A_\e\cap B_\e)\geq C \e^d$ by \eqref{nocciolino}. As a consequence,
 \[ \bbP_\l \bigl( \xi ( {\L}_{n+\e }\setminus {\stackrel{\circ}{\L}} _{n-\e} )= 0\,|\, A_\e \cap B_\e) =1 +o(1)\,.
 \]
 By the above observation and Lemma \ref{osare}, in the r.h.s. of \eqref{57} we can replace the event $\{ \bar 0\lr S_n\text{ in }\cG\}$ with the event $\{  0\lr S_n\text{ in }\cG_{\e}^*\} $ with an error $o(1)$. In particular} to get \eqref{rondine} it is enough to show 
that 
\begin{equation}\label{upupa}
\bbP_\l (0\lr   S_n\text{ in }\cG_{\e}\,|\,A_\e\cap B_{\e})= \bbP_\l (0\lr   S_n\text{ in }\cG_{\e}^*\,|\,A_\e\cap B_{\e})+o(1)\,.
\end{equation}

Since the events $A_\e,B_\e$ do  not depend on $\s$, and since the random variables of $\s$--type  are i.i.d. w.r.t. $\bbP_\l$ \rosso{conditioned to  $\xi$},   we conclude that $\cG_{\e}^*$  and $\cG_{\e}^{\#}$ have the same law under 
$\bbP_\l ( \cdot |A_\e\cap B_\e)$. Hence, in order to prove \eqref{upupa} it is enough to show that
\begin{equation}\label{upupabis}
\bbP_\l (0\lr   S_n\text{ in }\cG_{\e}\,|\,A_\e\cap B_{\e})= \bbP_\l (0\lr   S_n\text{ in }\cG_{\e}^{\#}\,|\,A_\e\cap B_{\e})+o(1)\,.
\end{equation}
Trivially, \eqref{upupabis} follows from Lemma \ref{uffa3} below. The result stated in Lemma \ref{uffa3} is stronger than what we need here (we do not need the term $ \xi ( \L_{\rosso{n+2}} )$ in the expectation), and it is suited for a further application in the next section.
\end{proof}
\begin{Lemma}\label{uffa3} It holds
$\bbE_\l \bigl[ \xi ( \L_{\rosso{n+2}} ) \mathds{1}(\cG_{\e}^{\#} \not = \cG_\e)\,|\,A_\e\cap B_{\e}\bigr]=o(1)$.
\end{Lemma}
\begin{proof} Recall definition \eqref{pioggia}.
Since the  graphs $\cG_{\e}$ and $\cG_{\e}^{\#}$ have the same vertex set $\cV_\e$,  by an union bound we can estimate 
\begin{equation}\label{natale}
\begin{split}
&
\bbE_\l \bigl[ \xi ( \L_{\rosso{n+2}} ) \mathds{1}(\cG_{\e}^{\#} \not = \cG_\e)\,|\,A_\e\cap B_{\e}\bigr]
\\
& \leq  \sum _{ x\prec_{\rm lex} y \text{ in } \G_\e} \bbE_\l \bigl[  \xi ( \L_{\rosso{n+2}} ) \mathds{1}\bigl( \varphi_x^\e=1\,,\; \varphi_y^\e=1\,, \varphi^\e_0=1\,,\;  \Psi _{x,y}^{\bar x, \bar y}\not = \Psi _{x,y}^{x,y }\bigr)\,|\, A_\e\cap B_\e
\bigr]\,.
\end{split}
\end{equation}
Note that  $\xi ( \L_{\rosso{n+2}} ) \leq \xi  \bigl( \L_{\rosso{n+2}} \setminus ( R_x^\e\cup R_y^\e\cup R_0^\e)\bigr )+3=:Z$ whenever $ \varphi_x^\e= \varphi_y^\e=\varphi_0^\e=1$.
 We  also observe that, under  $\bbP_\l\bigl(\cdot \,|\, A_\e\cap B_\e \bigr)$, the random variables 
$  Z$, $ \varphi_x^\e$,
$ \varphi_y^\e$, $\mathds{1}\bigr( \Psi _{x,y}^{\bar x, \bar y}\not = \Psi _{x,y}^{x,y }\bigr)$ are independent.
As a consequence, \eqref{natale} implies that 
\begin{equation}\label{epifania}
\begin{split}
& \bbE_\l \bigl[ \xi ( \L_{\rosso{n+2}} ) \mathds{1}(\cG_{\e}^{\#} \not = \cG_\e)\,|\,A_\e\cap B_{\e}\bigr]
\leq  c\,   \sum _{ x\in \G_\e} \sum_{ y \in  \G_\e  \setminus\{0,x\} }
 p_\l(\e)^{2-\d_{0,x}}
P \bigl(\Psi _{x,y}^{\bar x, \bar y}\not = \Psi _{x,y}^{x,y } \bigr)  \,,
\end{split}
\end{equation}
where $\d_{0,x}$ denotes the Kronecker delta.

Above we have used   the definition of $p_\l(\e)$ given in \eqref{pierpi},  the fact
  that 
 \begin{equation}
 \bbP_\l \bigl(\varphi_z^\e=1\,|\, A_\e\cap B_\e
\bigl) 
= \begin{cases}1 & \text{ if }  z =0\,,\\ p_\l(\e)  & \text{ if } z\in \G_\e\setminus\{0\}\,,
\end{cases}
\end{equation}
and the estimate (recall that $\bbP_\l (A_\e^c|B_\e)=o(1)$)
\[ \bbE_\l\bigl[ Z\,|\, A_\e\cap B_\e\bigr]\leq  \frac{\bbE_\l\bigl[ Z \mathds{1}( B_\e)\bigr] }{\bbP_\l (A_\e\cap B_\e)}= 
 \bbE_\l\bigl[ Z\bigr]   \frac{ \bbP_\l ( B_\e)}{\bbP_\l (A_\e\cap B_\e)}= \frac{\bbE_\l\bigl[ Z\bigr] }{\bbP_\l (A_\e|B_\e)}=O(1)\,.
\]

From now on we distinguish between the RC model and the MA model.

\smallskip

\noindent
$\bullet$  We consider  the RC model.  Recall that the  \rosso{connection} function $g$ is good. To simplify the notation 
 we restrict  to   $m=2$ and  $r_2=1$ in Definition \ref{dea} (the general case is similar).
For $i=1,\,2$, we set 
\[
 \o_i ( \d) := \sup \{ |g(a) -g(b)|\,:\, a, b \in (r_{i-1},r_i) \text{ and 
} |a-b| \leq \d\}\,.\]
Since $g$ is uniformly continuous in $(r_{i-1},r_{i})$ we know that
  $\o_i(\d)\rgh0$ when $\d\rgh 0$, for $i=1,\,2$. 
Since $g$ has support inside $(0,1)$,  $P \bigl(\Psi _{x,y}^{\bar x, \bar y}\not = \Psi _{x,y}^{x,y } \bigr) =0$ if $|x-y|\geq 1$ and $|\bar x-\bar y|\geq 1$. By taking $\e$ small, this always happens if $|x-y|\geq 2$. Hence, in the sum inside \eqref{natale} we can restrict to $x,y$ with $|x-y|<2$.

As a consequence we have \begin{equation} \label{lip1}
\begin{split}
\text{r.h.s. of \eqref{epifania}}
&\leq c \sum_
{x\in \G_\e
}\sum_{
\substack{
y\in \G_\e\setminus\{0,x\}:\\
 |x-y|<2}} p_\l(\e)^{2-\d_{0,x}}  |g(|x-y|)-g(|\bar x-\bar y|)|\,.
\end{split}
\end{equation}

\medskip
It remains to prove  that  the r.h.s. of \eqref{lip1} is $o(1) $.

Since  $|z-\bar z|<\sqrt{d}\e$ for all $z\in \G_\e$, 
$|x-y|$ differs from $|\bar x-\bar y|$ by at most $2 \sqrt{d}{\e}$.
We set  $M_{x,y}:=\max\{|x-y|,|\bar x-\bar y|\}$ and $m_{x,y}:=\min\{|x-y|,|\bar x-\bar y|\}$.  Note that $m_{x,y}>0$.
Since $g$ has support inside $(0,1)$, in \eqref{lip1} we can restrict to the case $m_{x,y}<r_2=1$.   
Moreover, if we consider the following cases:
\begin{itemize}
\item[(i)] $0=r_0< m_{x,y}< M_{x,y}<  r_1$,
\item[(ii)] $r_1< m_{x,y}<M_{x,y}< r_2=1$,
\end{itemize}  
we can bound
\[  |g(|x-y|)-g(|\bar x-\bar y|)| \leq \o_1 ( 2\sqrt{d}\e ) \text{ in the case (i)},\]
\[  |g(|x-y|)-g(|\bar x-\bar y|)| \leq \o_2 ( 2\sqrt{d}\e ) \text{ in the case (ii)}.\]
As a consequence the contribution in the r.h.s. of \eqref{lip1} of the pairs $x,y$ with $x\not =0$ and $M_{x,y}, \,m_{x,y}$ which satisfy case (i) or (ii), is bounded by $ c\e^{-2d} p_\l (\e)^2 \bigl( \o_1 ( 2\sqrt{d}\e )+\o_2 ( 2\sqrt{d}\e )\bigr)= O\bigl( \o_1 ( 2\sqrt{d}\e )+\o_2 ( 2\sqrt{d}\e )\bigr)=o(1)$. Similarly the  contribution in the r.h.s. of \eqref{lip1} of the pairs $x=0,y$ with  $M_{x,y}, \,m_{x,y}$ which \rosso{satisfy} case (i) or (ii), is $O\bigl( \o_1 ( 2\sqrt{d}\e )+\o_2 ( 2\sqrt{d}\e )\bigr) =o(1)$. 
\\
The other pairs $x,y$ in \eqref{lip1} we have not considered yet satisfy (a) $m_{x,y}\leq r_1\leq M_{x,y}$ or (b) $m_{x,y} \leq r_2\leq M_{x,y}$. 
Defining $r:=r_1$ in case (a) and $r:=r_2$ in case (b), we can restrict to study the contribution in the r.h.s. of \eqref{lip1} of the pairs $x,y$ which satisfy $m_{x,y}\leq r\leq M_{x,y}$. We now estimate such contribution.  Since 
$m_{x,y} \geq |x-y| - 2\sqrt{d}{\e}$ and $M_{x,y} \leq  |x-y| + 2\sqrt{d}{\e}$, it must be 
\begin{equation}\label{tg2}
  r-  2\sqrt{d}{\e} \leq |x-y|  \leq r+ 2\sqrt{d}{\e}\,.\end{equation}
The number of points $y\in \G_\e$ satisfying \eqref{tg2} are of order $O(\e^{-d+1})$, hence the pairs $x,y$ with $m_{x,y} \leq  r\leq M_{x,y}$ and $x\not =0$ are of order $O(\e^{-2d+1})$ while the pairs $x,y$ with $m_{x,y} \leq r\leq M_{x,y}$ and $x=0$ are $O(\e^{-d+1})$. Bounding in both cases 
$|g(|x-y|)-g(|\bar x-\bar y|)|
$ by $1$, we conclude that  the contribution in  the r.h.s. of \eqref{lip1} of the pairs $x,y$ with $m_{x,y}\leq r\leq M_{x,y}$ is 
bounded by $ O(\e^{-2d+1}) p_\l(\e)^2 + O(\e^{-d+1})p_\l(\e) = o(1)$.

\medskip

\noindent
$\bullet$  We consider  the MA model. 
As for \eqref{lip1} we have 
\begin{equation} \label{salina100}
\begin{split}
\text{r.h.s. \eqref{epifania}} 
\leq \sum_
{x\in \G_\e
}\sum_{
\substack{
y\in \G_\e\setminus\{0\}\,,\\
 |x-y|<2}}    p_\l(\e)^{2-\d_{0,x} }  \bigl|  h(|x-y|)- h( |\bar x-\bar y| ) \bigr|\,.
\end{split}
\end{equation}
Since by assumption $h$ is good,  one can proceed exactly as done for the RC model and conclude that  the r.h.s. of \eqref{salina100} is \rosso{of} order $o(1)$.
\end{proof}

\subsection{Conclusion of the proof of \eqref{uri1} in Proposition \ref{prop_discreto}} By combining  Lemmas
\ref{uffa1} and \ref{uffa2}  we get that $\tilde \theta_n(\l)=\bbP_\l (0\lr   S_n\text{ in }\cG_{\e}\,|\,A_\e\cap B_{\e})+o(1) $. On the other hand, by construction  the  random graph $\cG_\e$ sampled according to $\bbP_\l  (\cdot\,|\,A_\e\cap B_{\e})$ has the same law of the random graph $G_\e$ sampled according to $\bbP_\l ^{(\e)}(\cdot\,|\, \eta^\e_0=1) $. This implies that 
\[  \bbP_\l (0\lr   S_n\text{ in }\cG_{\e}\,|\,A_\e\cap B_{\e}) = \bbP_\l ^{(\e)}(  0\lr   S_n  \,|\, \eta^\e_0=1)  = \tilde \theta _n ^{(\e)}(\l)\,.
\]
This completes the proof of \eqref{uri1} for $n=k$. As stated at the beginning, the choice $n=k$ was to simplify the notation, the proof is the same for general $k$.

\section{Proof of \eqref{uri2} in Proposition \ref{prop_discreto}}\label{sec_discreto2}
We use the same convention on constants $c,c_1, c_2 \dots $,  on  $O(\e^\a)$ and $o(1)$ as \rosso{stated at the beginning of  the previous section}. 
Below we restrict to $n\geq 1$.

\smallskip

Due to  \eqref{rospo} we need to prove that 
\begin{equation}\label{kanan}  \bbE_{0,\l} \bigl[ | \text{Piv}_+ ( 0 \lr  S_n ) | \bigr]\   =   \lim _{\e \downarrow 0} \bbE^{(\e)}_\l \bigl[ \bigl|  {\rm Piv}(0 \lr  S_n ) \bigr| \bigr] \end{equation}
(recall the definition of $ \text{Piv}_+ ( 0 \lr  S_n )$ given after \eqref{rana}).
We define the function $g:\bbR^d \times \cN \to \bbR$ as 
\[ g(x, \xi)=
\mathds{1}(x\in R_0^{\e}) \int _\Sigma
P(d\s) 
\bigl |\, \text{Piv}_+(0\lr  S_n)( \xi, \s) \, \bigr  |
\,.\]
Then, by the property of the Palm distribution and of $P$  (cf. \rosso{\cite[Thm. 12.2.II and Eq. (12.2.4)]{DVJ} and}  Section \ref{solitario}), 
\begin{equation}\label{scintilla}
\begin{split}
\l\e^d\bbE_{0,\l}[|\text{Piv}_+(0\lr  S_n)|]& = \l E_{0,\l}   \Big[ \int _{\bbR^d}\rosso{dx} g(x,\xi)\Big ] = \bbE_\l\Big[\int_{\bbR^d}\xi(dx)g(x,\t_x\xi )\Big]=\\
&=\bbE_\l\Big[\int_{R_0^{\e}}\xi(dx)|\text{Piv}_+(x\lr  S_n(x))|\Big]\,.
\end{split}
\end{equation}
We recall that $S_n(x)=S_n+x$.
We can write  the last member of \eqref{scintilla}  as $\cC_1+\cC_2$, 
with $\cC_1$ and $\cC_2$ defined below. 
We set $N_\e:=\xi(R_0^{\e})$. Then, using independence and that $N_\e$ is a Poisson r.v. with parameter $\l \e^d$, we get 
\begin{equation}
\label{x1}
\begin{split}
\cC_1&:=\bbE_\l \Big[\int_{R_0^{\e}}\xi(dx)|\text{Piv}_+(x\lr  S_n(x))|\mathds{1}(N_\e\geq 2) \Big]\\ &\leq 
\bbE_\l\bigl[\xi(\L_{n+\rosso{2}})\,N_\e\mathds{1}(N_\e\geq 2)\bigr]=\bbE_\l\bigl[(N_\e+\xi(\L_{n+\rosso{2}}\setminus R_0^{\e}))\,N_\e \mathds{1}(N_\e\geq 2)\bigr]
\\ &\leq \bbE_\l\bigl[N_\e^2\mathds{1}(N_\e \geq 2) \bigr]+c_1 \bbE_\l\bigl [N_\e\mathds{1}(N_\e \geq 2) \bigr]\leq c_2\bbE_\l \bigl [N_\e^2 \mathds{1}(N_\e \geq 2) \bigr]
\\ &=c_2(\bbE_\l \bigl[N_\e^2\bigr]-\bbP_\l \bigl[N_\e =1 \bigr])
=c_2(\l\e^d+\l^2\e^{2d}-\l\e^de^{-\l\e^d})=O(\e^{2d}).
\end{split}
\end{equation}
\begin{Remark}\label{vetri}For  the first inequality in \eqref{x1} we point out that, given $x\in R_0^\e\cap \xi $,  the set $\text{Piv}_+(x\lr  S_n(x))$ (referred to  $\cG$) must be  contained in $\xi \cap \L_{n+2}$. Indeed, if we take a path in $\cG$ from $x$ to the complement of $x+(-n,n)^d$ and call $y$ the first vertex of the path outside   $x+(-n,n)^d$, then the euclidean distance between  $y$ and  $x+(-n,n)^d$ is  smaller  than $1$ (recall that all edges in $\cG$ have length smaller than $1$). In particular, we have that $\|y\| _\infty < \|x\|_\infty + n +1 \leq n+2$. As a consequence, to know if   $x\lr  S_n(x)$ in $\cG$ (or in the graph obtained by removing from $\cG$ a vertex $z$ and the  edges containing $z$), it is enough to know the vertexes of $\cG$ inside $\L_{n+2}$ and the edges formed by these vertexes.
\end{Remark}

We now bound the remaining contribution $\cC_2$:
\begin{equation}
\label{x2}
\begin{split}
\cC_2&:=\bbE_\l\Big[\int_{R_0^{\e}}\xi(dx)|\text{Piv}_+(x\lr  S_n(x))|\mathds{1}(N_\e =1)\Big]\\ &
=\bbE_\l \bigl[|\text{Piv}_+(\bar 0\lr  S_n(\bar 0))| \mathds{1}(N_\e =1)  \bigr]=\bbE_\l \bigl [|\text{Piv}_+(\bar 0\lr  S_n(\bar 0))|   \mathds{1}(N_\e =1)   \mathds{1}(A_\e)  \bigr]\\ &
+\bbE_\l \bigl[|\text{Piv}_+(\bar 0\lr  S_n(\bar 0))| \mathds{1}(N_\e =1)   \mathds{1}(A^c_\e)\bigr]\,.
\end{split}
\end{equation}
We note that (see also the computation of $ \bbE_\l \bigl [N_\e^2 \mathds{1}(N_\e \geq 2) \bigr]
$ in \eqref{x1})
\begin{equation}
\label{x3}
\begin{split}
\bbE_\l &  \bigl[|\text{Piv}_+(\bar 0\lr  S_n(\bar 0))|  \mathds{1}(N_\e =1) 
 \mathds{1}(A^c_\e)\bigr]
\leq \bbE_\l\bigl [\xi(\L_{n+\rosso{2}}) \mathds{1}(N_\e =1)   \mathds{1}(A^c_\e)\bigr ]
\\ &\leq\sum_{y\in\G_{\e}\setminus\{0\}}\bbE_\l \bigl[ \xi(\L_{n+\rosso{2}}) \mathds{1}(N_\e =1) 
\mathds{1}\bigl(\xi ( R^{\e}_y) \geq  2\bigr) \bigr] 
\\ &\leq \sum_{y\in\G_{\e}\setminus\{0\}}\bbE_\l \bigl[\xi(\L_{n+\rosso{2}}\setminus(R^{\e}_y\cup R_0^{\e}))\bigr]\bbP_\l \bigl(N_\e =1\bigr)\bbP_\l\bigl( \xi ( R^{\e}_y) \geq 2  \bigr)
\\ &+2\sum_{y\in\G_{\e}\setminus\{0\}}\bbE_\l \bigl[ \xi(R^{\e}_y)\mathds{1}( \xi( R_y^{\e}) \geq 2) \bigr]
\bbP_\l (N_\e=1)\leq \sum_{y\in\G_{\e}\setminus\{0\}}O(\e^{3d})=O(\e^{2d}).
\end{split}
\end{equation}
Since \eqref{scintilla}$=\cC_1+\cC_2$, by \eqref{x1},\eqref{x2} and \eqref{x3}, we get (note that $B_\e=\{N_\e=1\}$)
\begin{equation}
\label{x4}
\bbE_{0,\l}[|\text{Piv}_+(0\lr  S_n)|]=\bbE_\l \bigl[|\text{Piv}_+(\bar 0\lr  S_n(\bar 0))|   \mathds{1}(A_\e\cap B_\e)\bigr] \cdot\frac{1}{\l\e^d}+o(1)\,.
\end{equation}

\rosso{In what follows, given one of our random  graphs on the grid $\G_\e$ as $G_\e$ (cf. Section \ref{poligono}), $\cG_\e$, $\cG_\e^*$ and $\cG_\e^{\#}$ (cf. Section \ref{sec_discreto1}), and given an event $A$ regarding the graph, we call $\text{Piv}_+(A)$ the set of vertexes  $x$ of the  graph for which the following property holds: the event $A$ is realized  by the graph under consideration, but it does not take place when removing from the graph the vertex $x$ and all edges containing the vertex $x$.}

\rosso{\begin{Lemma}\label{fontana}
It holds
\begin{multline}\label{x5} 
\bbE_\l  \bigl[|\text{Piv}_+(\bar 0\lr  S_n(\bar 0))|   \mathds{1}(A_\e\cap B_\e)\bigr] 
= \\
\bbE_\l \bigl[|\text{Piv}_+( 0\lr  S_{n} \text{ in }\blu{\cG_{\e}^{*}}) |   \mathds{1}(A_\e\cap B_\e)\bigr]+o(\e^d)\,.
\end{multline}
\end{Lemma}}
\begin{proof}
\rosso{We can bound
\begin{equation}
\begin{split}
& \bbE_\l  \bigl[|\text{Piv}_+(\bar 0\lr  S_n(\bar 0))|   \mathds{1}(A_\e\cap B_\e)\mathds{1}
\bigl(
\xi ( \L_{n+\e }\setminus {\stackrel{\circ}{\L}} _{n-\e} ) \geq 1
\bigr)
\bigr] \\
&\leq
\bbE_\l  \bigl[ \xi(\L_{n+2})  \mathds{1}(A_\e\cap B_\e)\mathds{1}
\bigl(
\xi ( \L_{n+\e }\setminus {\stackrel{\circ}{\L}} _{n-\e} ) \geq 1
\bigr)
\bigr]\\
&\leq  \bbE_\l  \bigl[ W   \mathds{1}( B_\e)\mathds{1}
\bigl(
\xi ( \L_{n+\e }\setminus {\stackrel{\circ}{\L}} _{n-\e} ) \geq 1
\bigr)
\bigr]+\bbE_\l  \bigl[\xi ( \L_{n+\e }\setminus {\stackrel{\circ}{\L}} _{n-\e} )   \mathds{1}( B_\e)\bigr]\\
& \leq c\, \bbP_\l\bigl (B_\e\bigr) \bbP_\l \bigl( \xi ( \L_{n+\e }\setminus {\stackrel{\circ}{\L}} _{n-\e} ) \geq 1 
\bigr)
+\bbE_\l  \bigl[\xi ( \L_{n+\e }\setminus {\stackrel{\circ}{\L}} _{n-\e} ) 
\bigr]\bbP_\l(B_\e)\\
&=O( \e^{d+1})=o(\e^d)\,,
\end{split}
\end{equation}
where $W:= \xi(\L_{n+2}\setminus \L_{n+\e } )+\xi ( {\stackrel{\circ}{\L}} _{n-\e}\setminus R_0^\e)+1 $ (note that the third inequality follows from the independence property of the Poisson point process).
As a consequence, \eqref{x5} follows by observing that 
\begin{multline}\label{x54} 
\bbE_\l  \bigl[|\text{Piv}_+(\bar 0\lr  S_n(\bar 0))|   \mathds{1}(A_\e\cap B_\e)\mathds{1}
\bigl(
\xi ( \L_{n+\e }\setminus {\stackrel{\circ}{\L}} _{n-\e} ) =0 
\bigr)
\bigr] 
= \\
\bbE_\l \bigl[|\text{Piv}_+( 0\lr  S_{n} \text{ in }\blu{\cG_{\e}^{*}}) |   \mathds{1}(A_\e\cap B_\e)\mathds{1}
\bigl(
\xi ( \L_{n+\e }\setminus {\stackrel{\circ}{\L}} _{n-\e} ) =0 
\bigr)
 \bigr]\,.
\end{multline}
}
\blu{Let us justify  the above observation.  We assume that event $A_\e\cap B_\e$ is fulfilled  and that $\xi( \L_{n+\e }\setminus {\stackrel{\circ}{\L}} _{n-\e})=0$. Recall that $\cG^*_\e$ is obtained by restricting the graph $\cG$ to $\L_{n+1}$ and by sliding any  vertex $\bar x $ to $x$. Since $S_n(\bar 0) \subset  \L_{n+\e }\setminus {\stackrel{\circ}{\L}} _{n-\e}$, 
if $\bar{0} \lr S_n(\bar{0})$ in $\cG$ then $\bar 0 \lr y$ for some point $y\in \L_{n-\e+1}\setminus \L_{n+\e} $ (using that edges have length smaller than $1$).
It must be $y= \bar{v}$ for some $v\in \G_\e$. Since $\|y-v\|_\infty \leq \e$, we conclude that $v\in  \L_{n+1}\setminus \L_{n} $.  Since we can restrict to  paths from $0 $ to $y$ with  intermediate points  lying inside $\L_{n-\e}$, we have that all the intermediate points  are of the form $\bar z $ for some $z\in \G_\e$. We therefore get that the above path realizing the event $\bar{0} \lr S_n(\bar{0})$ in $\cG$  corresponds to a path in $\cG_\e^*$ from $0 $ to $v$, $\|v\|_\infty \geq n$. On the other hand,  since $\xi( \L_{n+\e }\setminus {\stackrel{\circ}{\L}} _{n-\e})=0$, any path in $\cG_\e^*$ from $0 $ to $v$, with $\|v\|_\infty \geq n$, is obtained by sliding  some path in $\cG$ from $\bar 0$ to $\L_{n+\e}^c$. As $S_n(\bar 0) \subset  \L_{n+\e }\setminus {\stackrel{\circ}{\L}} _{n-\e}$, these paths in $\cG$ must realize the event $\bar{0}\lr S_n(\bar 0)$. This correspondence between paths  implies  a correspondence between $(+)$--pivotal points, leading to  identity \eqref{x54}.}
\end{proof}
In the last term in \eqref{x5} \blu{we can replace $\cG^*_\e$ with $\cG_{\e}^{\#}$, since they have the same law under $\bbP_\l$ conditioned to $\xi$. Now} we would like to replace $\cG_{\e}^{\#}$ with  $\cG_{\e}$. This is  possible  due to
Lemma \ref{uffa3}. Indeed, we have  
\begin{equation}\label{marsiglia}
\bbP_\l( A_\e\cap B_\e)= \bbP_\l (B_\e) [1-\bbP_\l( A_\e^c|B_\e)]= \bbP_\l (B_\e)  (1+o(1)) = \l \e^d (1+o(1))\,,
\end{equation}
 thus implying that Lemma \ref{uffa3} is equivalent to the property 
\begin{equation}\label{nivea}
\bbE_\l \bigl [
\xi (\L_{n+\rosso{2}})    \mathds{1}(A_\e\cap B_\e) 
\mathds{1}(\cG_{\e}\neq\cG_{\e}^{\#}) ]=o(\e^d)\,.
\end{equation}
By combining \eqref{x4}, \eqref{x5}, \eqref{marsiglia} and \eqref{nivea} we conclude that
\begin{equation}\label{adm_holdo}
\bbE_{0,\l}[|\text{Piv}_+(0\lr  S_n)|]=\bbE_\l \bigl[|\text{Piv}_+( 0\lr  S_{n} \text{ in }\cG_{\e}) |  \,|\, A_\e\cap B_\e \bigr]+o(1)\,.
\end{equation}
Due to the definition of the graph $G_\e$ built on $\bigl(\O_\e, \bbP^{(\e)}_\l\bigr)$ we have 
\begin{equation}\label{eroina}
\bbE_\l \bigl[|\text{Piv}_+( 0\lr  S_{n} \text{ in }\cG_{\e}) |  \,|\, A_\e\cap B_\e \bigr]=\bbE_\l ^{(\e)} \bigl[|\text{Piv}_+( 0\lr  S_{n} ) |  \,|\, \eta_0^\e=1 \bigr]\,.
\end{equation}
Above, and in what follows, events appearing in $\bbE_\l ^{(\e)}, \bbP_\l ^{(\e)}$ are referred to the graph $G_\e$.

\smallskip

By combining \eqref{adm_holdo} and \eqref{eroina} we have achieved  that
\begin{equation}\label{force}
\bbE_{0,\l}[|\text{Piv}_+(0\lr  S_n)|] =\lim _{\e \downarrow 0} \bbE_\l ^{(\e)} \bigl[|\text{Piv}_+( 0\lr  S_{n} ) |  \,|\, \eta_0^\e=1 \bigr]\,\rosso{.}
\end{equation}
To derive \eqref{kanan} from \eqref{force} it is enough to apply the following result:
\begin{Lemma}
\label{x9}
It holds
 $ \bbE_\l ^{(\e)} \bigl[|{\rm Piv}_+( 0\lr  S_{n} ) |  \,|\, \eta_0^\e=1 \bigr]
 =  \bbE^{(\e)}_\l \bigl[ \bigl|  {\rm Piv}(0 \lr  S_n ) \bigr| \bigr]$. 
\end{Lemma}
\begin{proof}
Using the fact that  $\{\text{Piv}_+(0\lr  S_n)\}$ is empty if  $\eta_0^{\e}\not =1$, we get
\begin{equation}
\begin{split}
 \bbE_\l ^{(\e)} \bigl[|{\rm Piv}_+( 0\lr  S_{n} ) |  \,|\, \eta_0^\e=1 \bigr]
& =\frac{1}{p_\l(\e)} \sum_{x\in \G_{\e}}\bbE_\l^{(\e)}\bigl[\mathds{1}\bigl( x\in\text{Piv}_+(0\lr  S_n) \bigr)\eta_0^{\e}\bigr]
\\
 &=
\frac{1}{p_\l(\e)} \sum_{x\in \G_{\e}}\bbE_\l ^{(\e)}\bigl[\mathds{1}\bigl( x\in\text{Piv}_+(0\lr  S_n) \bigr)\bigr]
\\
&=\frac{1}{p_\l(\e)} \sum_{x\in \G_{\e}}\bbE_\l^{(\e)} \bigl[\mathds{1}\bigl( x\in\text{Piv}(0\lr  S_n) \bigr)\eta_x^\e \bigr] \,.
\end{split}
\end{equation}Since 
 the events $\{x\in\text{Piv}(0\lr  S_n)\}$ and $\{\eta_x^{\e}=1\}$ are independent, the last expression equals
$ \sum_{x\in \G_{\e}}\bbE^{(\e)}_\l\bigl[\mathds{1}\bigl( x\in\text{Piv}(0\lr  S_n) \bigr) \bigr] 
 $, thus concluding the proof.
 \end{proof}
%
%
%


\bigskip
\bigskip

\noindent
{\bf Acknowledgements}. The authors thank H. Duminil-Copin and R. Meester for useful discussions. \rosso{We also thank the anonymous referee for the careful  reading  of the manuscript and for his/her corrections and suggestions}.



\begin{thebibliography}{2}

\bibitem{AHL} V. Ambegoakar, B.I. Halperin, J.S. Langer. \emph{Hopping conductivity in disordered systems}. Phys, Rev B {\bf 4}, 2612-2620 (1971).







\bibitem{DVJ} D.J. Daley, D.  Vere--Jones. \emph{An Introduction to the theory of point processes}. New York, 
Springer Verlag, 1988.


  
  
  
 
  
  
  
   \bibitem{DRT1} H. Duminil-Copin, A. Raoufi, V. Tassion. \emph{Sharp phase transition for the random--cluster and Potts models via decision trees}. Preprint  arXiv:1705.03104 (2017).
   
\bibitem{DRT2}   H. Duminil-Copin, A. Raoufi, V. Tassion. \emph{Exponential decay of connection probabilities for subcritical Voronoi percolation in $\bbR^d$}. \rosso{To appear on Prob. Theory Rel. Fields. Available online (2018).}
   
   
   \bibitem{FM}
 A.\ Faggionato, P.\ Mathieu. \emph{Mott law as upper bound for a random
walk in a random environment}. Comm. Math. Phys. {\bf 281}, 263--286
(2008). 


 \bibitem{FM_phys} 
 \rosso{A.\ Faggionato}. \emph{Mott's law for the effective conductance of the Miller--Abrahams random resistor network}. In preparation. A preliminary version is available online as arXiv:1712.07980.

\bibitem{FSS} A.\ Faggionato, H.\ Schulz--Baldes, D.\ Spehner. {\em
     Mott law as lower bound for a random walk in a random
     environment.} Comm.\ Math.\ Phys., {\bf 263},
     21--64 (2006).



\bibitem{G} G. Grimmett,  \emph{Percolation}. 	Die Grundlehren der mathematischen Wissenschaften {\bf 321}. Second edition, Springer Verlag, Berlin, 1999.

\bibitem{M}  R. Meester, R. Roy.  \emph{Continuum percolation}. Cambridge Tracts in Mathematics {\bf 119}. First edition, Cambridge University Press, Cambridge, 1996.

\bibitem{MA}  A. Miller, E. Abrahams. \emph{Impurity Conduction at Low Concentrations}. Phys. Rev. {\bf 120}, 745--755 (1960).




\bibitem{JZG} J. Jiang, S. Zhang, T. Guo; \emph{Russo's formula, uniqueness of the infinite cluster, and continuous differentiability of free energy for continuum percolation}. J. Appl. Prob. {\bf 48}, 597--610 (2011)


\bibitem{OSSS}  R. O'Donnell, M. Saks, O. Schramm,  R. Servedio; \emph{Every decision tree has an influential
variable}. FOCS, 2005.

\bibitem{POF} M. Pollak, M. Ortu{\~n}o, A. Frydman. \emph{The electron glass}. First edition, Cambridge University Press, United Kingdom, 2013.

\bibitem{SE}
S. Shklovskii, A.L.  Efros. \emph{Electronic Properties of Doped Semiconductors}. Springer Verlag, Berlin, 1984.


\bibitem{Z} \rosso{S. Ziesche. 
\emph{Sharpness of the phase transition and lower bounds for the critical intensity in continuum percolation on $\bbR^d$}.
Ann. Inst. H. Poincar\'e Probab. Statist. {\bf 54},
 866-878 (2018).}
 
 
\end{thebibliography}
\end{document}